\documentclass[12pt]{amsart}
\usepackage{amssymb}
\title{Highest weight $\sl_2$-categorifications II: structure theory}
\author{Ivan Losev}
\address{Department
of Mathematics, Northeastern University, Boston MA 02115 USA}
\email{i.loseu@neu.edu}
\thanks{Supported by the NSF grants DMS-0900907, DMS-1161584}
\renewcommand{\sl}{\mathfrak{sl}}
\newcommand{\Cat}{\mathcal{C}}
\newcommand{\OCat}{\mathcal{O}}

\newcommand{\C}{\mathbb{K}}
\newcommand{\Q}{\mathbb{Q}}
\newcommand{\Z}{\mathbb{Z}}
\newcommand{\wt}{\operatorname{wt}}
\newcommand{\End}{\operatorname{End}}

\newcommand{\cont}{\operatorname{cont}}
\newcommand{\Hom}{\operatorname{Hom}}

\newcommand{\Ext}{\operatorname{Ext}}

\newcommand{\head}{\mathsf{head}}
\newcommand{\res}{\operatorname{res}}
\newcommand{\gl}{\mathfrak{gl}}

\newcommand{\GL}{\operatorname{GL}}

\newcommand{\Fun}{\mathcal{F}}
\newcommand{\Eun}{\mathcal{E}}
\unitlength=1mm   \numberwithin{equation}{section}
\newtheorem{Thm}{Theorem}[section]
\newtheorem{Prop}[Thm]{Proposition}

\newtheorem{Lem}[Thm]{Lemma}
\theoremstyle{definition}

\newtheorem{Rem}[Thm]{Remark}

\numberwithin{table}{section} \oddsidemargin=0cm
\begin{document}
\begin{abstract}
This paper continues the study of highest weight categorical $\sl_2$-actions initiated in part I. We
start by refining the definition given there and showing that all examples considered in part I
are also highest weight categorifications in the refined sense. Then we prove that
any highest weight $\sl_2$-categorification can be filtered in such a way that the successive quotients
are so called basic highest weight $\sl_2$-categorifications. For a basic highest weight categorification
we determine minimal projective resolutions of standard objects. We use this, in particular, to examine the structure
of tilting objects in basic categorifications and to show that the Ringel duality is given by
the Rickard complex. We apply some of these structural results to categories $\mathcal{O}$
for cyclotomic Rational Cherednik algebras.
\end{abstract}
\maketitle
\tableofcontents
\section{Introduction}
Categorical $\sl_2$-actions (=$\sl_2$-categorifications) were introduced by Chuang and Rouquier in \cite{CR}
to establish derived equivalences for blocks of the symmetric groups in positive characteristic.
In \cite{cryst}  we have introduced the notion of a highest weight $\sl_2$-categorification
and used that to describe crystal structures for many classical highest weight categories appearing in Representation theory:
the categories of rational and polynomial representations of $\GL$, the parabolic categories $\mathcal{O}$
of type $A$ and the categories $\mathcal{O}$ over the cyclotomic Rational Cherednik algebras. In this
paper we are going to study the structural features of highest weight $\sl_2$-categorifications.

The definition of a highest weight $\sl_2$-categorification should incorporate some compatibility conditions
between the $\sl_2$-categorification and the highest weight structure. The conditions that appeared in \cite[4.1]{cryst}
can be divided into two groups: the compatibility of the action with standard objects (axioms (HWC0),(HWC2) in loc.cit.)
and also the compatibility of the action with an ordering on the category (axioms (HWC1),(HWC3),(HWC4)). In this
paper we will essentially keep (HWC0),(HWC2) but we will need to modify the compatibility with orderings.
Namely, we will define so called {\it hierarchy structures} on posets and check that posets of basically all
highest weight categories of interest can be equipped with such structures. In our new definition of a highest
weight $\sl_2$-categorification we will require the poset of a highest weight category to admit a hierarchy
structure.

Let us sketch an easy, yet very important, example coming from the Lie representation theory. Namely, consider
the BGG category $\mathcal{O}$ for $\gl_n(\mathbb{C})$ and the sum of its blocks with $\rho$-shifted
highest weights of the form $(x_1,\ldots,x_n), x_i=0$ or $1$.
Denote the sum by $\Cat$. Its poset is the set $\{+,-\}^n$, where to a highest weight $(x_1,\ldots,x_n)$
we assign the $n$-tuple $t\in \{+,-\}^n$ with $t_i=+$ if $x_i=0$ and $t_i=-$ if $x_i=1$. The categorification
functors $E,F$ come from the tensor products with the $\gl_n$-modules $\mathbb{C}^n, (\mathbb{C}^n)^*$ followed by
taking projections to appropriate blocks.    As an $\sl_2$-module,  the Grothendieck group  $[\Cat]$ of $\Cat$ is
identified with $(\Q^2)^{\otimes n}$,  while the classes of standards are
the monomial elements. Roughly speaking, a highest weight categorification with the latter property will be called
{\it basic}. %It is not clear to us at the moment whether any basic categorification is isomorphic to $\Cat$.

There is a reason why we call such categorifications basic: any highest weight categorification can be filtered
in such a way that the subsequent quotients are basic categories, see Subsection \ref{SS_family_filtr}.
More precisely, in an $\sl_2$-categorification $\Cat$ one can consider a filtration $0\subset\Cat_{\leqslant 0}\subset\Cat_{\leqslant 1}
\subset \ldots =\Cat$ by Serre subcategories such that all $\Cat_{\leqslant i}$ are stable with respect to the categorical action. Then
the subsequent quotients $\Cat_i:=\Cat_{\leqslant i}/\Cat_{\leqslant i-1}$ carry $\sl_2$-categorifications.
On the other hand, in a highest weight category $\Cat$ we can consider a filtration $\Cat_{\leqslant i}$ by Serre subcategories
such that the labeling set $\Lambda_{\leqslant i}$ is a poset ideal. Then $\Cat_{\leqslant i}/\Cat_{\leqslant i-1}$
has a natural highest weight structure. In a highest weight $\sl_2$-categorification $\Cat$ we can find a filtration $\Cat_{\leqslant i}$
such that both properties hold and  the subsequent quotients $\Cat_{\leqslant i}/\Cat_{\leqslant i-1}$
are basic highest weight $\sl_2$-categorifications. This can be regarded as a highest weight analog of
\cite[Theorem 5.24, Remark 5.25]{CR}.

A filtration as in the previous paragraph allows to reduce some questions about general highest weight categorifications
to basic ones. In particular, in Proposition \ref{Prop:E_stand} we will describe the heads of the objects of the form $E\Delta(\lambda)$, where $\Delta(\lambda)$ is a standard object. A solution for this problem is known in some special cases, see, for example, \cite{BK_functors} and our answer can be regarded as a generalization of that. Our most important result about the structure of basic categorifications
is a description of minimal projective resolutions of standard objects. We will see, in particular, that the description is the same as
for the example of a basic categorification described above (where it is classical).

Let us describe the structure of this paper. In Section \ref{S_hier} we introduce a combinatorial structure -- a hierarchy --
that a poset of a highest weight $\sl_2$-categorification will be supposed to have. We equip some classical posets, such
as parabolic highest weights or multipartitions, with hierarchy structures. We also introduce the notion of a dual hierarchy
structure. We need this because there is a natural naive duality for $\sl_2$-categorifications (swapping the categorification
functors $E$ and $F$) that does not preserve a hierarchy structure but rather maps it into its dual.

In Section  \ref{S_hw_cat} we (re)introduce highest weight $\sl_2$-categorifications. These are highest weight
categories, whose posets are equipped with a (dual) hierarchy structure, such that the categorification functors
are compatible with the highest weight structure on the category and with the hierarchy structure on the poset.
Then we show that the examples that have already appeared in \cite{cryst} are highest weight categorifications
in this new sense as well.

In Section \ref{S_cat_spl} we introduce an important technical tool to study highest weight $\sl_2$-\!\! categorifications:
categorical splitting. This is a categorical version of a splitting structure on the poset that is a part of
a hierarchy structure. We refer the reader to the beginning of Section \ref{S_cat_spl} for details. Using the
categorical splitting we prove the filtration result, Proposition \ref{Prop:filtration}, mentioned above.

Section \ref{S_proj_resol} is a central part of this paper. There we  determine, Theorem \ref{Thm_proj_resol}, a minimal projective resolution of a standard object in a basic categorification, equivalently, compute the dimensions of the exts between standard
and irreducible objects. As an application we determine the head of an object of the form $E\Delta(\lambda)$,
Proposition \ref{Prop:E_stand}. Also we deduce some information about the indecomposable summands of $E P(\lambda), F P(\lambda)$
for general categorifications. Finally, we describe the structure of the objects of the form $E L(t), F L(t)$,

In Section \ref{S_Ringel} we study the Ringel duality for a basic categorification. We show that the Ringel dual
of a (basic) highest weight $\sl_2$-categorification is again a (basic) highest weight $\sl_2$-categorification.
This allows us to deduce the information
about tilting objects from the known information about projectives. Further, we show
that, in a basic categorification, the reflection functor (=the Rickard complex, see \cite[6.1]{CR})
actually performs the Ringel duality. \footnote{In a joint work with B. Webster, \cite{LW},
 we have recently proved that a basic categorification is unique. So, in characteristic 0, results of Sections \ref{S_proj_resol} and \ref{S_Ringel} give just new proofs of the aforementioned facts for the standard basic categorifications. I do not know whether,
in positive characteristic, these results were known previously.}
We also compute the Hom spaces between the standard objects.

In Section \ref{S_Cher_appl} we provide some applications to the categories $\mathcal{O}$
for cyclotomic Rational Cherednik algebras.

In the final section of this paper we will list some open problems.

{\bf Acknowledgements.} I would like to thank R. Bezrukavnikov, J. Brundan, S. Cautis, M. Feigin,  S. Griffeth, J. Kamnitzer, A. Kleshchev, A. Lauda, C. Stroppel,  and B. Webster for stimulating discussions. I also would like to thank the referee for many remarks on the previous
version of this paper.

\section{Notation} Here we gather some notation used in the paper.

Let $\Cat$ be  a highest weight category with poset $\Lambda$. By $\Cat^\Delta$ we denote the full subcategory of
standardly filtered objects of $\Cat$. By $\Cat$-$\operatorname{proj}$ (resp., $\Cat$-$\operatorname{tilt}$) we denote
the subcategories of projective (resp., tilting) objects in $\Cat$. These are subcategories of $\Cat^\Delta$.
Also by $\Cat^\nabla$ we denote the full subcategory of costandardly filtered objects. We write $\Cat^{opp}$
for the opposite category of $\Cat$.  We write $[\Cat]$ for the rational form of the Grothendieck group of $\Cat$.

For $\lambda\in \Lambda$ let $\Delta(\lambda),\nabla(\lambda),L(\lambda),P(\lambda),T(\lambda)$
denote the standard, costandard, simple, projective, tilting objects corresponding to $\lambda$.

For a functor $\varphi$ let $\varphi^*,\varphi^!$ denote its right and left adjoint, respectively.

\section{Hierarchy structures}\label{S_hier}
\subsection{Definition}\label{SS_hier_def}
Let $\Lambda$ be a  poset. Recall that a subset $I$ in a poset $\Lambda$ is called
an {\it ideal}, if $\lambda\in I$ and $\mu\leqslant \lambda$ implies $\mu\in I$. By a {\it coideal}
one means the complement of an ideal.

A {\it hierarchy} structure on $\Lambda$ will be a collection of additional structures.
The first one, a {\it family} structure, has already appeared in \cite{cryst}, this is a collection of triples $(\Lambda_a,n_a,\sigma_a)$
indexed by elements $a$ of some indexing set $\mathfrak{A}$. Here $\Lambda_a$ is a subset of $\Lambda$ (to be called a {\it family}),
$n_a$ is a non-negative integer, and  $\sigma_a$ is a bijection $ \{+,-\}^{n_a}\xrightarrow{\sim} \Lambda_a$.
We require $\Lambda=\bigsqcup_{a\in \mathfrak{A}}\Lambda_a$.  We equip $\{+,-\}^{n_a}$ with the following dominance ordering: $(t_1,\ldots,t_n)\leqslant (t_1',\ldots,t_n')$ if, for all $m$, the
number of $+$'s among $t_1,\ldots,t_m$ is bigger than or equal to the number of $+$'s among $t_1',\ldots,t_m'$
and the total number of $+$'s in $t$ and $t'$ is the same (so the maximal elements
are of the form $-\ldots-+\ldots+$, while the minimal elements are of the form $+\ldots+-\ldots-$).
We require that $\sigma_a^{-1}:\Lambda_a\rightarrow \{+,-\}^{n_a}$ is increasing for any $a$.

For the future use, let us remark that the family structure gives rise to an $\sl_2$-action on $\mathbb{Q}^\Lambda$.
Namely, let $v_+,v_-$ be a basis of $\mathbb{Q}^2$ such that the $\sl_2$-action is given by $ev_+=v_-, fv_-=v_+, ev_-=fv_+=0$.
We can identify $\mathbb{Q}^{\Lambda_a}$ with $(\mathbb{Q}^2)^{\otimes n_a}$ by mapping the basis vector in $\mathbb{Q}^{\Lambda_a}$
corresponding to $t\in \Lambda_a$ to $v_{t_1}\otimes v_{t_2}\otimes\ldots \otimes v_{t_n}$. Since $\mathbb{Q}^{\Lambda}=\bigoplus_{a\in \mathfrak{A}}\mathbb{Q}^{\Lambda_a}$, we get an $\sl_2$-action on $\mathbb{Q}^{\Lambda}$.

Let us consider  a relatively simple example of a family structure, more examples will be provided in the next subsection.
Set $\Lambda:=\mathcal{P}$, the poset of partitions with respect to the dominance ordering:
we say that $\lambda<\mu$ if $|\lambda|=|\mu|$ (where as usual, $|\lambda|$ is the number partitioned by $\lambda$),
$ \lambda\neq\mu$ and $\sum_{i=1}^k \lambda_i\leqslant \sum_{i=1}^k \mu_i$
for each $k$. Modulo $|\lambda|=|\mu|$ the last condition is equivalent to
$\sum_{i\geqslant k}\lambda_i\geqslant \sum_{i\geqslant k}\mu_i$.
Let $N$ be a non-negative integer. Pick a residue $r$ modulo $N$ (if $N=0$, then $r$ is
just an integer). By an $r$-box we mean a box $(x,y)$ (where $x$ is the number of a row and $y$ is the number of
a column) whose content $y-x$ is congruent to $r$ modulo $N$. We depict the partitions in the ``French style'':
they are located in the positive quadrant. 

Define a family structure on $\Lambda$ as follows, compare to \cite[4.2]{cryst}: two partitions $\lambda_1,\lambda_2$ lie in the same family if the partitions obtained from $\lambda_1,\lambda_2$ by removing all removable $r$-boxes coincide; for a family $\Lambda_a$
the map $\sigma_a$ is obtained by reading all addable/removable boxes from bottom to top, writing a + for an addable
box and a $-$ for a removable one. For example, if $N=3, r=1$, and $\lambda=(5,3^4)$, where the superscript means
the multiplicity, then the $\sigma_a(\lambda)=--+$.

In the general setting, the second structure will be a collection of  partitions of $\Lambda$,
one for each $a\in \mathfrak{A}$.
Namely, fix $a\in \mathfrak{A}$. To $a$ we assign a partition $\Lambda=\Lambda_{<}^a\sqcup \underline{\Lambda}^a_-\sqcup\underline{\Lambda}^a_+\sqcup \Lambda_{>}^a$. We require such
partitions to satisfy the following axioms.

\begin{itemize}
\item[(S0)] For each $a$, the subsets $\Lambda^a_{<}, \Lambda^a_<\sqcup\underline{\Lambda}^a_-, \Lambda^a_<\sqcup\underline{\Lambda}^a_-\sqcup\underline{\Lambda}^a_+$ are poset ideals in $\Lambda$.
\item[(S1)] $n_a=0$ if and only if $\underline{\Lambda}^a_-, \underline{\Lambda}^a_+=\varnothing$.
\item[(S2)] For each  $a,b$ the family $\Lambda_b$ is contained either in $\Lambda^a_{<}$ or
in $\Lambda^a_{>}$ or in $\Lambda^a_=:=\underline{\Lambda}^a_-\sqcup \underline{\Lambda}^a_+$.
Moreover, suppose $\Lambda_b\subset \Lambda_=^a$. An element $\lambda\in \Lambda_b$ is contained in
$\underline{\Lambda}^a_?$ if and only if the rightmost entry of $\sigma^{-1}_b(\lambda)$ is $?$ (for $?=+,-$).
\item[(S3)] Let $a,b\in \mathfrak{A}$. If $\Lambda_b\subset \Lambda_=^a$, then $\underline{\Lambda}^b_?=
\underline{\Lambda}^a_?$ for $?=+,-$, and $\Lambda^a_{<}=\Lambda^b_{<}, \Lambda^a_{>}=\Lambda^b_{>}$. The inclusion $\Lambda_b\subset \Lambda^a_>$ holds if and only if
$\Lambda_a\subset \Lambda^b_<$.
\item[(S4)] Let $a\in A$. Then there is a (automatically, unique)
poset isomorphism $\iota: \underline{\Lambda}^a_-\rightarrow \underline{\Lambda}^a_+$
that maps $\underline{\Lambda}^a_-\cap \Lambda_b$ to $\underline{\Lambda}^a_+\cap \Lambda_b$ such that if $\sigma_b^{-1}(\lambda)=t-$
for $t\in \{+,-\}^{n_b-1}$, then $\sigma_b^{-1}(\iota(\lambda))=t+$.
\end{itemize}

The assignment $a\mapsto (\Lambda^a_<, \underline{\Lambda}^a_-,\underline{\Lambda}^a_+,\Lambda^a_>), a\in \mathfrak{A},$ will be called
a {\it splitting structure}.

Let us illustrate this by the example of $\Lambda=\mathcal{P}$ as above.
Define a splitting structure on $\Lambda$ as follows. For an integer $k$ let $|\lambda|^k$ denote the number of
boxes in $\lambda$ with content $k$. Define a new ordering $\prec$ on $\mathcal{P}$: $\lambda\prec\mu$
if there exists $k$ with $|\lambda|^{l}=|\mu|^{l}$ for all $l<k$ and $|\lambda|^k>|\mu|^k$.
We remark that $\lambda<\mu$ implies $\lambda\prec\mu$.

Pick a family $\Lambda_a$ and let $(x,y)$ be the top-most
addable/removable $r$-box for this family with content, say, $m$.
We remark that $\lambda_x$ can have one of the two values, say $s,s+1$.
Also for $\lambda\in\Lambda_a$ the numbers $|\lambda|^k$ do not depend
on the choice of $\lambda$ as long as $k<m$.

Let $\Lambda^a_{>}$ consist of all partitions $\mu$ such that
\begin{itemize}
\item[(i)] either there is $k<m$ with $|\lambda|^l=|\mu|^l$ for all $l<k$ and $|\lambda|^k>|\mu|^k$,
\item[(ii)] or $|\lambda|^l=|\mu|^l$ for all $l<m$ and $\mu_x<s$,
\item[(iii)] or $|\lambda|^l=|\mu|^l$ for all $l<m, \mu_x=\mu_{x-1}=s$.
\end{itemize}
Let $\Lambda^a_=$ consist of all partitions with $|\lambda|^l=|\mu|^l$ for all $l>m$,
while $\mu_x=s,s+1$, $\mu_{x-1}>s$. We partition $\Lambda^a_=$ into the union $\underline{\Lambda}^a_+\sqcup
\underline{\Lambda}^a_-$ according to the value of $\mu_x$: a partition $\mu$ is in
$\underline{\Lambda}^a_+$ if and only if $\mu_x=s$. Finally, let $\Lambda^a_<$ consist of the remaining
partitions.

It is easy to see that (S0) is satisfied for $\prec$ and hence for $<$ too. (S1) is
straightforward. Let us check (S2). It is easy to see that $\Lambda^a_=$ is a union of families
(exactly those, where $(x,y)$ is a top addable/removable box). It remains to verify that
$\Lambda^a_>$ is a union of families. So let $\lambda\in \Lambda_a,
\mu\in \Lambda^a_{>}$ with $k$ being as in the previous paragraph. Assume (i) holds.
Let $x'$ be the largest number such that the $x'$th row in $\mu$
contains a box with content $k$. Of course, $x'>x$, in particular, $\lambda$ has
no addable/removable $r$-boxes in rows $x'$ or higher. Then $\mu_{x''}=\lambda_{x''}$ for $x''>x'$
and $\mu_{x'}<\lambda_{x'}$. %If $\mu'_{x'}=\lambda_{x'}$
%for all $\mu'$ lying in the same family as $\mu$, then $r-k$ is not divisible by $N$. From here it is easy to
%see that $\mu'\in \Lambda^a_{>}$. So we can assume that  $\lambda_{x'}>\mu_{x'}$.
It follows that any addable/removable box in $\mu$ lying in
the rows with numbers $>x'$ is also addable/removable for $\lambda$. As a consequence,
there are no addable/removable $r$-boxes in $\mu$ in rows with numbers $> x'$.
Also there is no addable $r$-box in the $x'$th row of $\mu$ provided $\lambda_{x'}=\mu_{x'}+1$
(otherwise, $\lambda$ has a removable box in that row).
Let $\nu$ be a partition in the same family as $\mu$. From the previous two sentences
it follows that $\nu_{x''}=\mu_{x''}$ for all $x''>x'$ and $\nu_{x'}\leqslant\mu_{x'}$ if $\lambda_{x'}=\mu_{x'}+1$. In the latter case,
$\nu\in \Lambda^a_{>}$.
If $\lambda_{x'}>\mu_{x'}+1$, then $\nu_{x'}\leqslant \mu_{x'}+1<\lambda_{x'}$. So we see that $\nu\in \Lambda^a_{>}$.
The case when (ii) holds for $\mu$ is analyzed in a similar way. Also (iii) itself specifies a union of families.
The remaining part of (S2) is easy to check.

(S3) and (S4) follow directly from the construction.
%\item or $|\lambda|^l=|\mu|^l$ for all $l>m$, while $\mu_x<N$.
%\end{itemize}

    %For a partition $\mu$ let $\mu^x$ denote the partition obtained from
%$\mu$ by taking all boxes in $\mu$ lying in rows $x+1$ and above. We remark that the partition
%$\lambda^x$ is the same for all $\lambda\in \Lambda_a$, say $\lambda'$, while  $\lambda_x$ takes one of the two consecutive
%values, say $m,m+1$. Let $\underline{\Lambda}^a_-$ (resp., $\underline{\Lambda}^a_+$) consist of all $\mu$ with $\mu^x=\lambda'$
%and $\mu_x=m+1$ (resp., $\mu_x=m$). Let $\Lambda^a_{>}$ consist of all partitions $\mu$
%such that
%\begin{enumerate}
%\item either $|\mu^x|<|\lambda'|$,
%\item or $|\mu^x|=|\lambda'|, \mu^x>\lambda'$,
%\item or $\mu^x=\lambda^x$  and $\mu_x<m$.
%\end{enumerate}
%Finally, let $\Lambda^a_<$ consist of all remaining partitions. The conditions (S0)-(S4) are straightforward to check.

Let us return to the general situation. Now let $\underline{\Lambda}^a$ be one of the isomorphic posets $\underline{\Lambda}^a_?$. It has a family structure induced
from $\Lambda$. Namely, for the families we take the non-empty intersections $\underline{\Lambda}_b:=\underline{\Lambda}^a_+\cap\Lambda_b$,
they are indexed by a subset $\underline{\mathfrak{A}}^a\subset \mathfrak{A}$. For $b\in \underline{\mathfrak{A}}^a$
we set $\underline{n}_b:=n_b-1$ and define a map $\underline{\sigma}_b:\underline{\Lambda}_b\rightarrow \{+,-\}^{\underline{n}_b}$ by
$\sigma_b(\lambda+)=\underline{\sigma}_b(\lambda)+$. On the other hand, the splitting structure on $\Lambda$ does not seem to define
any splitting structure on $\underline{\Lambda}^a$. This is why we need the next piece of a structure.

A {\it hierarchy} on $\Lambda$ is a collection  $\mathfrak{H}$ of pairs $(\mathfrak{A}', \Lambda(\mathfrak{A}'))$, where   $\mathfrak{A}'\subset\mathfrak{A}$ and $\Lambda(\mathfrak{A}')$ is a poset
with family and splitting structures.   Also we require that the following
axioms hold:
\begin{itemize}
\item[(H0)] If $(\mathfrak{A}', \Lambda(\mathfrak{A}')), (\mathfrak{A}'', \Lambda(\mathfrak{A}''))\in \mathfrak{H}$
and $\mathfrak{A}'=\mathfrak{A''}$, then $\Lambda(\mathfrak{A}')=\Lambda(\mathfrak{A}'')$. Further, either one of
the subsets $\mathfrak{A}',\mathfrak{A}''$ is contained in the other, or $\mathfrak{A}',\mathfrak{A}''$ are disjoint.
\item[(H1)] $(\mathfrak{A},\Lambda)\in \mathfrak{H}$. If $(\mathfrak{A}', \Lambda(\mathfrak{A}'))\in \mathfrak{H}$, then, for any $a\in\mathfrak{A}'$, there is a splitting structure on $\underline{\Lambda(\mathfrak{A}')}^a$ such that
$\left((\underline{\mathfrak{A}}')^a, \underline{\Lambda(\mathfrak{A}')}^a\right)\in \mathfrak{H}$. Moreover,
every element $(\mathfrak{A}'',\Lambda(\mathfrak{A}''))$ is obtained from $(\mathfrak{A},\Lambda(\mathfrak{A}))$
by doing several steps as in  the previous sentence.
%Let $\mathfrak{A}'\subset \mathfrak{A}''$ be such that %$(\mathfrak{A}',\Lambda(\mathfrak{A}')),(\mathfrak{A}'',\Lambda(\mathfrak{A}''))\in\mathfrak{H}$.
%    Then there is a finite chain of subsets
%$\mathfrak{A}_0\subset \mathfrak{A}_1\subset \ldots\subset \mathfrak{A}_k$ from $\mathfrak{H}$ with
%$\mathfrak{A}'=\mathfrak{A}_0, \mathfrak{A}''=\mathfrak{A}_k$ subject to the following condition:
%for any $i=0,1,\ldots,k-1$ there is $a_i\in \mathfrak{A}_{i+1}$ such that
%$\mathfrak{A}_i=\underline{\mathfrak{A}}_{i+1}^{a_i}$ and $\Lambda(\mathfrak{A}_i)=\underline{\Lambda}(\mathfrak{A}_{i+1})^{a_i}$
%as posets with family structures.
\item[(H2)] Any descending chain of embedded subsets in $\mathfrak{H}$ terminates.
\end{itemize}

We remark that this definition is given in such a way that any $\underline{\Lambda}^a$ comes equipped with a
hierarchy structure induced from $\Lambda$.

To produce a hierarchy structure for $\Lambda=\mathcal{P}$ we need to repeatedly define the splitting structures on the emerging posets
of the form $\underline{\Lambda}^a$.  Take the set $\underline{\Lambda}^a_-$
and declare that in all partitions in this set the box in the position $(x,y)$ as above is {\it frozen}. Then we repeat the construction in the previous paragraph and take the topmost unfrozen addable/removable box $(x',y')$. To define the next layer of
the hierarchy we will freeze $(x',y')$ too, and so on. Clearly, (H2) is satisfied. %Consider the mapping $\underline{\Lambda}^a\rightarrow \Lambda$ given by $\lambda\mapsto %\lambda^{\bar{x}}$. Define the splitting structure on $\underline{\Lambda}$ by pulling back that from $\Lambda$.
%To produce the subsequent layers of the hierarchy we proceed similarly. The conditions (H0),(H1) are satisfied
%automatically.

\subsection{Examples}\label{SS_hier_ex}
Let us start with a very easy, ``basic'' so to say, example when we only have one family and $\Lambda=\Lambda_a=\{+,-\}^n$.
The sets $\underline{\Lambda}^a_-,\underline{\Lambda}^a_+$
are introduced in a unique possible way. The poset $\underline{\Lambda}$ is just $\{+,-\}^{n-1}$ and a
hierarchy structure is introduced inductively. We  have $\Lambda^a_{>}=\Lambda^a_{<}=\varnothing$.

The example given in the previous subsection can be generalized to multipartitions. Let $\ell$ be a positive integer,  $p=(\kappa, s_0,\ldots,s_{\ell-1})$ be a collection of complex numbers, $\kappa$ being non-integral. Consider the set $\mathcal{P}_\ell$ of $\ell$-multipartitions $\lambda=(\lambda^{(0)},\ldots,\lambda^{(\ell-1)})$. A box in a multipartition $\lambda$
is given by a triple $(x,y,i)$, where $i=0,1,\ldots,\ell-1$ is the number of a multipartition, where the box occurs, and $(x,y)$
are its coordinates: $x$ is the row number, and $y$ is the column number. To a box $\beta=(x,y,i)$ we assign its {\it shifted content} $\cont(\beta)=y-x+s_i$. We say that boxes $\beta,\beta'$ are equivalent and write $\beta\sim \beta'$
if $\cont(\beta)-\cont(\beta')\in \kappa^{-1}\mathbb{Z}$. Also to a box $\beta=(x,y,i)$ we assign the number
$d^p(\beta)=\kappa \ell \cont(\beta)-i$. We write $\beta\leqslant\beta'$ if $\beta\sim \beta'$ and $d^p(\beta)- d^p(\beta')$
is a non-negative integer. For two elements $\lambda,\mu\in \mathcal{P}_{\ell}$ we write $\lambda\leqslant \mu$ if $|\lambda|=|\mu|$
and we can number boxes $b_1,\ldots,b_n$ of $\lambda$ and $b_1',\ldots,b_n'$ of $\mu$ in such a way that $b_i\leqslant b_i'$ for all $i$.

%For two boxes $\beta'=(x',y',i'),
%\beta''=(x'',y'',i'')$ we say that $\beta'<\beta''$ if either $\cont(\beta')<\cont(\beta'')$ or $\cont(\beta')=\cont(\beta'')$
%and $i'<i''$. For a box $\beta=(x,y,i)$ let $[\beta]$ be the pair $(\cont(\beta),i)$. The elements $[\beta]$  form
%a discrete  linearly ordered set.

%For $\lambda\in \mathcal{P}_\ell$ and a box $\beta$  (either in $\lambda$ or outside of $\lambda$) let $|\lambda|^{[\beta]}$ be the number of %boxes $\beta'$ in $\lambda$ such that $[\beta']=[\beta]$. For $\lambda,\mu\in \mathcal{P}_\ell$ with $\lambda\neq \mu, |\lambda|=|\mu|$ we %set $\lambda<_{p}\mu$ if $\sum_{[\beta']\leqslant [\beta]}|\lambda|^{[\beta']}\geqslant \sum_{[\beta']\leqslant [\beta]}|\mu|^{[\beta']}$ for %all classes $\beta$. We set $\lambda\prec_p\mu$ if there is $[\beta]$ with $|\lambda|^{[\beta']}=|\mu^{[\beta']}|$ for all
%$[\beta']<[\beta]$, while $|\lambda|^{[\beta]}>|\mu|^{[\beta]}$. Clearly, $\lambda<_p\mu$ implies $\lambda\prec_p\mu$.

A family structure on $\mathcal{P}$ already appeared in \cite[4.2]{cryst}.
Namely, for a nonzero complex number $z$, we call a box $\beta$ a $z$-box if
$\exp(2\pi \kappa \cont(\beta)\sqrt{-1})=z$. Clearly, the $z$-boxes form an equivalence class
with respect to $\sim$. As before, two multipartitions $\lambda_1,\lambda_2$
lie in the same family (relative to $z$, below we will also use the name ``$z$-family'')
if the multipartitions obtained from $\lambda_1,\lambda_2$ by removing
all removable $z$-boxes are the same.  As we remarked in loc.cit., all addable/removable $z$-boxes have
distinct numbers $d^p(\beta)$ and all of these numbers differ from each other by an integer. To get the map
$\sigma_a^{-1}(\lambda)$ we read addable/removable $z$-boxes $\beta$ of $\lambda$ in the increasing
order with respect to $d^p(\beta)$ and write a $+$ if the box is addable and a $-$ if the box is removable.

Let us define a splitting structure, more or less similar in spirit to that on the usual partitions introduced in the
previous subsection. Pick a family $\Lambda_a$. Let $\beta=(x,y,i)$ be the common smallest addable/removable
$z$-box for the multipartitions in this family. For a multipartition $\lambda$ and a
box $\beta'$, let $|\lambda|^{\beta'}$ denote the number of boxes $\beta''\in \lambda$ with
$\beta''\sim\beta', d^p(\beta'')=d^p(\beta')$ -- these are precisely the boxes lying in the same
diagram and in the same diagonal as $\beta'$. For all boxes $\beta'$ with $\beta'\not\sim \beta$
or with $\beta'<\beta$ the numbers $|\lambda|^{\beta'}$ do not depend on the choice of $\lambda\in \Lambda_a$.

Let $B_1,\ldots,B_t$ be all equivalence classes of boxes that can appear in a multipartition of $|\lambda|$ for some
 $\lambda\in \Lambda_a$ with $B_t$ being the class of $z$-boxes.
Let $\Lambda^a_{>}$ consist of all multipartitions $\mu$ such that there is a box $\beta'$ with the  following three
properties
\begin{itemize}
\item $\beta'\in B_i$ with $i<t$ or  $\beta'<\beta$.
\item $|\mu|^{\beta''}=|\lambda|^{\beta''}$ for all $\beta''$ lying in $B_j$ with $j<i$
or  $\beta''<\beta'$.
\item $|\mu|^{\beta'}<|\lambda|^{\beta'}$.
\end{itemize}
Let $\Lambda^a_=$ consist of all multipartitions $\mu$ such that $|\mu|^{\beta''}=|\lambda|^{\beta''}$
for all boxes $\beta''$ that either lie in $B_i$ with $i<t$ or  $\beta''<\beta$.
Then automatically $\beta$ is an addable/removable box in any $\mu\in \Lambda^a_=$ and we form
the subsets $\underline{\Lambda}^a_+, \underline{\Lambda}^a_-$ accordingly. Finally, let $\Lambda^a_<$
consist of the remaining partitions. The proof that (S0)-(S4) hold is similar (and actually easier) to the one given
in the previous subsection. We would like to remark that $\Lambda^a_==\Lambda_a$.

The hierarchy structure is introduced in a way similar to the above: by freezing addable/\! removable boxes.
The condition (H2) is easily seen to be satisfied.

Let us consider one more example: parabolic highest weights. Namely, we fix $m>0$ and positive integers
$s_1,\ldots,s_{\ell}$ with $\sum_{i=1}^\ell s_i=m$. Let $\Lambda$ consist of all sequences $A=(a_1,\ldots,a_m)$
of integers $a_1,\ldots,a_m$ such that $a_1>a_2>\ldots>a_{s_1}, a_{s_1+1}>\ldots>a_{s_1+s_2},\ldots,
a_{s_1+s_2+\ldots+s_{\ell-1}+1}>\ldots>a_m$. We say that $A<A'$ if there positive roots $\alpha_1,\ldots\alpha_k$
in the root system of type $A_{m-1}$ such that $A'=A+\alpha_1+\ldots+\alpha_k$.

Let us introduce a family structure on $\Lambda$ that essentially has already appeared in \cite{cryst}.
Namely, pick a non-negative integer $N\neq 1$  and a residue $r$ modulo $N$ (if $N=0$, then $r$ is to be thought
as an integer). A family equivalence relation is defined as follows: $A'\sim A''$ if for any index $j=1,\ldots,m$
exactly one of the following holds:
\begin{itemize}
\item[(i)] $a'_j=a''_j$.
\item[(ii)] $a'_j=a''_j+1$ and $a''_j\equiv r\operatorname{mod} N$.
\item[(iii)] $a''_j=a'_j+1$ and $a'_j\equiv r\operatorname{mod} N$.
\end{itemize}
A map $\sigma_a:\{+,-\}^{n_a}\rightarrow \Lambda_a$ is constructed as follows. Let $j_1<j_2<\ldots<j_{n_a}$
be all indexes $j$ such that the family contains elements $A',A''$ such that (ii) or (iii) holds for $j$.
Then for $A\in \lambda_a$ let $t=\sigma_a^{-1}(A)$  be defined as follows: $t_i=+$ (resp, $t_i=-$) if $a_{j_i}\equiv r$
(resp., $a_{j_i}\equiv r+1$) modulo $N$.

A splitting structure is defined similarly to what we had above. Namely, in the notation of the previous paragraph,
let $j=j_{n_a}$. Pick $A=(a_1,\ldots,a_m)\in \Lambda_a$. The values $a_{j+1},\ldots,a_m$ do not depend on
the choice of $A$, while $a_j$ takes one of the two values, say $s,s+1$. Let $\Lambda^a_{>}$ consist of
all $A'\in \Lambda, A'=(a_1',\ldots,a_m')$ such that
\begin{itemize}
\item either there is $j'>j$ such that $a_{j'}>a'_{j'}, a'_{j'+1}=a_{j+1},\ldots, a'_m=a_m$,
\item or $a_{j+1}=a'_{j+1},\ldots, a_m=a'_m$ and $a'_{j}<s$,
\item or $a'_{j-1}=s+1, a'_j=s$, while $a_{j+1}=a'_{j+1},\ldots, a_m=a'_m$
and there is $l$ such that $s_1+\ldots+s_{l-1}+1<j\leqslant s_1+\ldots+s_l$.
\end{itemize}

$\Lambda_=^a, \underline{\Lambda}^a_\pm$ and $\Lambda^a_<$ are introduced similarly to the case of partitions.
Checking (S0)-(S4) and introducing the hierarchy structure is completely analogous to the above.

\subsection{Dual hierarchy structures}\label{SS_h_dual}
In the sequel we will use also the notion of a {\it dual hierarchy structure}: basically looking at the leftmost
element in $\sigma_a^{-1}(\lambda)$ instead of the rightmost one. Let $\Lambda$ be a poset equipped with a
family structure with families $\Lambda_a, a\in \mathfrak{A}$. By a dual splitting structure we mean an assignment
that to each $a\in \mathfrak{A}$ assigns a splitting $\Lambda=\bar{\Lambda}^a_{>}\sqcup \underline{\bar{\Lambda}}^a_+\sqcup
\underline{\bar{\Lambda}}^a_-\sqcup \bar{\Lambda}^a_{<}$. This assignment is subject to the following axioms.
\begin{itemize}
\item[($\bar{\text{S}}$0)] For each $a$, the subsets $\bar{\Lambda}^a_{<}, \bar{\Lambda}^a_{<}\sqcup\underline{\bar{\Lambda}}^a_+, \bar{\Lambda}^a_{<}\sqcup\underline{\bar{\Lambda}}^a_+\sqcup\underline{\bar{\Lambda}}^a_-$ are poset
    ideals.
\item[($\bar{\text{S}}$1)] $n_a=0$ if and only if $\underline{\bar{\Lambda}}^a_+, \underline{\bar{\Lambda}}^a_-=\varnothing$.
\item[($\bar{\text{S}}$2)] For each  $a,b$ the family $\Lambda_b$ is contained either in $\bar{\Lambda}^a_{<}$ or
in $\bar{\Lambda}^a_{>}$ or in $\bar{\Lambda}^a_=:=\underline{\bar{\Lambda}}^a_+\sqcup \underline{\bar{\Lambda}}^a_-$.
Moreover, suppose $\Lambda_b\subset \bar{\Lambda}_=^a$. An element $\lambda\in \Lambda_b$ is contained in
$\underline{\bar{\Lambda}}^a_?$ if and only if the leftmost element of $\sigma_b^{-1}(\lambda)$ is $?$ (for $?=+,-$).
\item[($\bar{\text{S}}$3)] Let $a,b\in \mathfrak{A}$. If $\Lambda_b\subset \bar{\Lambda}_=^a$, then $\underline{\bar{\Lambda}}^b_?=
\underline{\bar{\Lambda}}^a_?$ for $?=+,-$ and $\bar{\Lambda}^a_{<}=\bar{\Lambda}^b_{<}, \bar{\Lambda}^a_{>}=\bar{\Lambda}^b_{>0}$. The inclusion $\Lambda_b\subset \bar{\Lambda}^a_>$ holds if and only if
$\Lambda_a\subset \bar{\Lambda}^b_<$.
\item[($\bar{\text{S}}$4)] Let $a\in A$. Then there is a poset isomorphism $\iota: \underline{\bar{\Lambda}}^a_+\rightarrow \underline{\bar{\Lambda}}^a_-$ that maps $\underline{\bar{\Lambda}}^a_+\cap \Lambda_b$ to $\underline{\bar{\Lambda}}^a_-\cap \Lambda_b$ such that if $\sigma_b^{-1}(\lambda)=+t$ for $t\in \{+,-\}^{n_b-1}$, then $\sigma_b^{-1}(\iota(\lambda))=-t$.
\end{itemize}

A definition of a dual hierarchy structure is now given by a complete analogy with that of a usual hierarchy structure.

As an example, let us introduce a  dual splitting structure on the poset of partitions $\mathcal{P}$. Instead the top-most removable
box in a family now we are going to consider the bottom-most one.
Namely, pick a family $\Lambda_a$ and let $(x,y)$ be the bottom-most
addable/removable box for this family with content, say, $m$.
We remark that $\lambda_x$ can have one of the two values, say $s,s+1$.
Also for $\lambda\in\Lambda_a$ the numbers $|\lambda|^k$ do not depend
on the choice of $\lambda$ as long as $k>m$.

Let $\bar{\Lambda}^a_{>}$ consist of all partitions $\mu$ such that
\begin{itemize}
\item[(i)] either there is $k>m$ with $|\lambda|^l=|\mu|^l$ for all $l>k$ and $|\lambda|^k<|\mu|^k$,
\item[(ii)] or $|\lambda|^l=|\mu|^l$ for all $l>m$ and $\mu_x>s+1$,
\item[(iii)] or $|\lambda|^l=|\mu|^l$ for all $l>m, \mu_x=\mu_{x+1}=s+1$.
\end{itemize}
Let $\bar{\Lambda}^a_=$ consist of all partitions with $|\lambda|^l=|\mu|^l$ for all $l>m$,
while $\mu_x=s,s+1$, $\mu_{x+1}\leqslant s$. We partition $\bar{\Lambda}^a_=$ into the union $\underline{\bar{\Lambda}}^a_+\sqcup
\underline{\bar{\Lambda}}^a_-$ according to the value of $\mu_x$: a partition $\mu$ is in
$\underline{\bar{\Lambda}}^a_+$ if and only if $\mu_x=s$. Finally, let $\bar{\Lambda}^a_<$ consist of the remaining
partitions.

In fact, one can formally obtain a dual hierarchy structure from a usual one. Namely, define, first, a dual family
structure on $\Lambda$. The decomposition $\Lambda=\bigsqcup_a \Lambda_a$ is the same as before. However, the map
$\sigma_a$ gets modified: we consider a new map $\bar{\sigma}_a: \{+,-\}^{n_a}\rightarrow \Lambda_a$ defined by
$\bar{\sigma}_a(t):=\sigma_a(\bar{t})$, where for $t=(t_1,\ldots,t_{n_a})$ we set
$\bar{t}:=(\bar{t}_{n_a},\bar{t}_{n_a-1},\ldots,\bar{t}_1)$ with $\bar{t}_i$ defined as the element
different from $t_i$. The splitting structure is the same but it now satisfies ($\bar{\text{S}}$0)-($\bar{\text{S}}$4)
and so is a dual splitting structure. Also the hierarchy structure stays the same but becomes a dual hierarchy
structure.

We would like to point out that the dual structure constructed on $\mathcal{P}$ in this way is \underline{different}
from what we have constructed just above, even the family structures are different. 
the topmost box. However, the two structures are isomorphic via the transposition
of Young diagrams (with reversing the order).

\section{Highest weight $\sl_2$-categorifications}\label{S_hw_cat}
\subsection{Reminder on highest weight categories}\label{SS_hw_reminder}
The goal of this subsection is to recall some basic facts about highest weight categories. Let $\C$ be a field.
Let $\Cat$ denote a  $\C$-linear  finite length abelian  category. Let $\Lambda$ be an indexing set
for the irreducible objects in $\Cat$, we write $L(\lambda)$ for the irreducible corresponding to $\lambda$.
Recall that by a highest weight structure on $\Cat$ one means a poset structure on $\Lambda$ together
with a collection of {\it standard} objects $\Delta(\lambda)$, one for each $\lambda\in \Lambda$, such that
the following axioms hold:
\begin{itemize}
\item[(HW1)] We have $\Hom(\Delta(\lambda),\Delta(\mu))=0$ if $\lambda>\mu$ and $\End(\Delta(\lambda))=\C$.
\item[(HW2)] For any $\lambda\in \Lambda$, there is a projective cover $P(\lambda)$ of $L(\lambda)$. It admits an epimorphism $P(\lambda)\twoheadrightarrow\Delta(\lambda)$ whose kernel is filtered with successive quotients $\Delta(\mu), \mu>\lambda$.
\end{itemize}

Let us recall that $\Cat^{opp}$ is also a highest weight category with respect to the poset $\Lambda$. Its standard
object (i.e., {\it costandard} objects in $\Cat$) $\nabla(\lambda)$ have the property $\dim\Ext^i(\Delta(\lambda),\nabla(\mu))=\delta_{i,0}\delta_{\lambda,\mu}$. Moreover, an object $M$ in $\Cat$
is standardly filtered if and only if $\Ext^1(M,\nabla(\lambda))=0$ for all $\lambda$. There is a similar characterization
of costandardly filtered objects. We have the BGG reciprocity:
the multiplicity of $\Delta(\lambda)$ in $P(\mu)$ equals to the multiplicity of $L(\mu)$ in $\nabla(\lambda)$.

\begin{Rem}\label{Rem:ordering}
We remark that if $\Cat$ is a highest weight category with respect to the orders $\leqslant_1,\leqslant_2$, then
it is so with respect to their intersection (as of subsets of $\Lambda\times\Lambda$). So there is the coarsest possible
ordering. It is generated by the relation $\leqslant$ given by $\lambda\leqslant \mu$ if $\Hom(\Delta(\lambda),\Delta(\mu))\neq 0$
or $\Delta(\mu)$ appears in $P(\lambda)$.
\end{Rem}

An object which is both
standardly filtered and costandardly filtered is called {\it tilting}. The indecomposable tiltings are indexed
by $\Lambda$, the tilting $T(\lambda)$ corresponding to $\lambda$ has the property that there is an inclusion
$\Delta(\lambda)\hookrightarrow T(\lambda)$ such that the quotient admits a filtration with successive quotients
of the form $\Delta(\mu),\mu<\lambda$.

Let us proceed to highest weight sub/quotient categories.

Let $\Lambda'$ be a poset ideal in $\Lambda$. Consider the Serre subcategory $\Cat'$ of
$\Cat$ spanned by $L(\lambda), \lambda\in \Lambda'$. We have the inclusion functor $\iota: \Cat'\rightarrow \Cat$.
This functor  has the left adjoint $\iota^!:\Cat\rightarrow \Cat'$ that takes the maximal quotient lying in $\Cat'$.

\begin{Lem}\label{Lem:hw_subcat}
\begin{itemize}
\item The category $\Cat'$ is a highest weight category with respect to the poset $\Lambda'$, the  standard objects
are $\Delta(\lambda), \lambda\in \Lambda'$, while the costandard objects are $\nabla(\lambda)$.
\item The functor $\iota^!$ is exact on $\Cat^\Delta$.
\item We have $\operatorname{Ext}_{\Cat'}^i(M,M')=\Ext^i_{\Cat}(M,M')$ for all $M,M'\in \Cat'$.
\end{itemize}
\end{Lem}
\begin{proof}
(1) and (2) are easy and standard. To show (3) we notice that to compute Ext's from $M$ to $M'$
one can replace $M$ with a standardly filtered complex and $M'$ with a costandardly filtered complex.
Then Ext's are computed via the total complex of  the corresponding double complex of Hom's that is independent
on whether it is taken in $\Cat$ or in $\Cat'$.
\end{proof}

Now consider the complement $\Lambda''=\Lambda\setminus \Lambda'$, a poset coideal. Set $\Cat'':=\Cat/\Cat'$
and let $\pi: \Cat\rightarrow \Cat''$ denote the quotient functor, we can identify $\Cat''$ with the category
of finitely generated modules over $\operatorname{End}(\bigoplus_{\lambda\in \Lambda''} P(\lambda))^{opp}$,
then $\pi=\Hom_{\Cat}(\bigoplus_{\lambda\in \Lambda''} P(\lambda),\bullet)$. The functor $\pi$ has left adjoint
$\pi^!$ given by taking the tensor product with $\bigoplus_{\lambda\in \Lambda''}P(\lambda)$ over
$\operatorname{End}(\bigoplus_{\lambda\in \Lambda''}P(\lambda))$.

\begin{Lem}\label{Lem:hw_quot}
In the above notation, the following holds.
\begin{itemize}
\item[(i)] The category $\Cat''$ is highest weight with respect to the poset $\Lambda''$, the standard
objects are $\pi(\Delta(\lambda)), \lambda\in \Lambda''$.
\item[(ii)] The functor $\pi^!$ defines an equivalence of $(\Cat'')^{\Delta}$ with the full subcategory
of $\Cat$ consisting of all objects that admit a filtration with successive quotients of the form
$\Delta(\lambda), \lambda\in \Lambda''$. A quasi-inverse equivalence is given by $\pi$.
\item[(iii)] For $M\in \Cat$, we have the following functorial exact sequence
$\pi^!\pi M\rightarrow M\rightarrow \iota\iota^! M\rightarrow 0$. Further, $\pi^!\pi, \iota\iota^!$
are exact endofunctors of $\Cat^\Delta$ and $\pi^!\pi M\hookrightarrow M$ for $M\in \Cat^\Delta$.
\end{itemize}
\end{Lem}

\subsection{Definition of a highest weight $\sl_2$-categorification}\label{SS_hw_def}
Let $\Lambda$ be a poset equipped with a hierarchy (and so, in particular, with family and splitting structures).
Let $\Cat$ be a split finite length $\C$-linear category
that is equipped with a categorical $\sl_2$-action, i.e., with biadjoint functors $E,F$ together with additional
structures, see \cite{CR}.
Also assume $\Cat$ is a highest weight category, whose standard objects $\Delta(\lambda)$ are indexed by the elements
of $\Lambda$.

One of the structures that enter the definition of an $\sl_2$-categorification is a decomposition
$\Cat=\bigoplus_{w\in \mathbb{Z}}\Cat_w$ according to the ``weight'' for the $\sl_2$-action. Conditions on
that decomposition are that $E\Cat_w\subset \Cat_{w+2}$, $F\Cat_{w}\subset \Cat_{w-2}$. Another part is a
pair of functor endomorphisms $X\in \operatorname{End}(E), T\in \operatorname{End}(E^2)$.
The condition on them is that there are $a,q\in \C$ with $a\neq 0$ if $q\neq 1$ such that
\begin{itemize}
\item $X-a$ is nilpotent.
\item The induced transformations $X_i=\operatorname{id}^{i-1}X\operatorname{id}^{n-i}, i=1,\ldots,n$ and
$T_j:=\operatorname{id}^{j-1}T\operatorname{id}^{n-j-1},$ $ j=1,\ldots,n-1,$ of $E^n$ satisfy the defining relations
of the affine Hecke algebra with parameter $q$ (the degenerate affine Hecke algebra when $q=1$).
\end{itemize}

One of the corollaries of the presence of $X,T$ is that the functors $E^n,F^n$ can be decomposed as $\C^{n!}\otimes_{\C} E^{(n)}, \C^{n!}\otimes_{\C} F^{(n)}$ for appropriate endofunctors $E^{(n)},F^{(n)}$ of $\Cat$.

We say that $\Cat$ is a highest weight categorification with respect to the hierarchy structure on $\Lambda$ if
\begin{itemize}
\item[(i)] The functors $E,F$ preserve $\Cat^\Delta$.
\item[(ii)] The map $[\Cat]\xrightarrow{\sim}\Q^{\Lambda}$ defined by $[\Delta(\lambda)]\mapsto \lambda$
(where in the right hand side we write $\lambda$ for the basis vector corresponding to $\lambda$)
is an isomorphism of $\sl_2$-modules, where the
$\sl_2$-action on $\Q^{\Lambda}$ was defined in the beginning of Subsection \ref{SS_hier_def}.
\end{itemize}
In other words, we require that for each $a\in \mathfrak{A}, t\in \{+,-\}^{n_a}$,
\begin{itemize}
\item there is a filtration
$0=\mathcal{F}_0\subset \mathcal{F}_1\subset\ldots\subset \mathcal{F}_r=E\Delta(\sigma_a(t))$ such that
$\mathcal{F}_i/\mathcal{F}_{i-1}=\Delta(\sigma_a(t^i))$, where $t^i\in \{+,-\}^{n_a}$ is determined
as follows. Let $j_1<j_2<\ldots<j_r$ be all indexes $j$ such that $t_{j}=+$, then $t^i$ is obtained from
$t$ by switching $t_{j_i}$ to a $-$.
\item there is a filtration $0=\mathcal{F}'_0\subset \mathcal{F}'_1\subset\ldots\subset \mathcal{F}'_s=F\Delta(\sigma_a(t))$ such that
$\mathcal{F}'_i/\mathcal{F}'_{i-1}=\Delta(\sigma_a(\bar{t}^i))$, where $\bar{t}^i\in \{+,-\}^{n_a}$ is determined
as follows. Let $j'_1>j'_2>\ldots>j'_s$ be all indexes $j$ such that $t_{j}=-$, then $\bar{t}^i$ is obtained from
$t$ by switching $t_{j'_i}$ to a $+$.
\end{itemize}
We remark that filtration subquotients can be placed in this order because $t^i> t^{i'}$  in $\{+,-\}^{n_a}$ whenever $i<i'$
and so $\sigma_a(t^i)\not< \sigma_a(t^{i'})$ in $\Lambda$ (and similarly in the other case).  The filtrations on  $E\Delta(\lambda), F\Delta(\lambda)$ with these subquotients (in this order) are unique, they will be called {\it standard}.
%\begin{itemize}
%\item[(i)] $E\Delta(\lambda)$ admits a filtration whose successive quotients are $\Delta(\lambda^1),\ldots,\Delta(\lambda^k)$,
%where the elements $\lambda^1,\ldots,\lambda^k$ are determined as follows. Set $t=\sigma_a^{-1}(\lambda)$ and let
%$j_1>j_2>\ldots>j_k$ be all indexes such that $t_{j_i}=+$. Then $\lambda^i:=\sigma_a(t^i)$, where $t^i$ is obtained
%from $t$ by replacing the $j_i$th element with a $-$.
%\item[(ii)] $F\Delta(\lambda)$ admits a filtration whose successive quotients are $\Delta(\bar{\lambda}^1),\ldots,\Delta(\bar{\lambda}^l)$,
%where the elements $\bar{\lambda}^1,\ldots,\bar{\lambda}^l$ are determined as follows. Set $t=\sigma_a^{-1}(\lambda)$ and let
%$j_1<j_2<\ldots<j_l$ be all indexes such that $t_{j_i}=-$. Then $\bar{\lambda}^i:=\sigma_a(\bar{t}^i)$, where $\bar{t}^i$ is obtained
%from $t$ by replacing the $j_i$th element with a $+$.
%\end{itemize}

Also we would like to point out   that this definition is different from \cite{cryst}. We still require the conditions (HWC0),(HWC2) from there but the remaining three conditions that were dealing with the poset structure  are now replaced by a (morally, much stronger) condition of having a hierarchy structure on the poset $\Lambda$. %Also we would like to remark that $\lambda^1<\ldots<\lambda^k$
%and $\bar{\lambda}^1<\ldots<\bar{\lambda}^l$  in $\Lambda$.

Similarly we can give a ``dual'' definition of a highest
weight $\sl_2$-categorification with respect to a dual hierarchy structure on $\Lambda$.

\begin{Rem}\label{Rem:opp_cat}
It is easy to
see that $\Cat^{opp}$ is a highest weight $\sl_2$-categorification with respect to the (dual)
hierarchy structure on $\Lambda$ provided $\Cat$ is. (i) follows  from the biadjointness of $E,F$ and the standard fact
that $M\in \Cat^\nabla$ if and only if $\Ext^1(M,N)=0$ for any $N\in \Cat^\Delta$. (ii) follows from the observation
that the multiplicity of $\nabla(\lambda')$ in $E\nabla(\lambda)$ equals $\Hom(\Delta(\lambda'), E\nabla(\lambda))=
\Hom(F\Delta(\lambda'), \nabla(\lambda))$ that coincides with the multiplicity of $\Delta(\lambda)$ in
$F\Delta(\lambda')$.
\end{Rem}

% the objects $\Delta(\lambda^1),\ldots,\Delta(\lambda^k)$ appear in this order if counted in $E\Delta(\lambda)$
%from top to bottom (i.e., $\lambda^1$ is the )

We will also impose a technical assumption on $\Cat$. Let us remark that if $\Cat'$ is a highest weight categorification
defined over a subfield $\C'\subset\C$, then we have a natural highest weight categorification on
$\Cat:=\C\otimes_{\C'}\Cat'$. In this case we say that $\Cat$ is defined over $\C'$.
We will suppose that at least one of the following holds:
\begin{itemize}
\item[(iii$^1$)] All blocks of $\Cat$ have a finite number of simples and $\C$ is infinite.
\item[(iii$^2$)] The field $\C$ is uncountable.
\item[(iii$^3$)] $\Cat$ is defined over a subfield of infinite codimension in $\C$.%, and $\C$ is algebraically closed.
\end{itemize}

Let us finish this subsection by explaining a naive duality for highest weight $\sl_2$-\!\! categorifications.
This duality will swap $E$ and $F$ and turn a hierarchy structure on $\Lambda$ into a dual hierarchy structure.
In more detail, consider a category $\bar{\Cat}$ that coincides with $\Cat$ as a highest weight category.
Set $\bar{E}:=F, \bar{F}:=E, \bar{\Cat}_w:=\Cat_{-w}$. Equip $\Lambda$ with the dual hierarchy structure
explained in Subsection \ref{SS_h_dual}. Clearly (i) and (ii) still hold, while neither of (iii$^1$)-(iii$^3$)
depended on the categorification structure at all. So we see that $\bar{\Cat}$ becomes a highest
weight categorification with respect to the dual hierarchy structure on $\Lambda$.

\subsection{Examples}\label{SS_hw_ex}
In this subsection we will consider some examples of categorifications that have already appeared
in \cite{cryst} and whose posets were equipped with hierarchy structures in Subsection \ref{SS_hier_ex}.
We will see that they are actually highest weight categorifications with respect to hierarchy structures.

First, consider the case when $\Lambda$ is the poset of parabolic highest weights, see Subsection \ref{SS_hier_ex}.
Suppose that $\C$ is an algebraically closed field of characteristic $0$  and the integer $N$ is $0$. Then $\Lambda$ is a poset
of (the integral block of) the parabolic category $\mathcal{O}$ for the Lie algebra $\gl_m$ and its parabolic subgroup with blocks of sizes
$s_1,\ldots,s_{\ell}$. It follows from  the construction of an $\sl_2$-categorification
on the parabolic category $\mathcal{O}$, see \cite{CR} or \cite{BK_shift}, that $\mathcal{O}$ satisfies
(i) and (ii). Also $\mathcal{O}$ satisfies  (iii$^1$) and (iii$^3$).
(iii$^1$) is a classical result. And (iii$^3$) follows from the observation that $\mathcal{O}$ is defined
over $\mathbb{Q}$.

%In particular, we can get a basic categorification as a special case of this construction
%(we need to take the usual category $\mathcal{O}$ and consider only the most singular blocks
%so that the only components of weights are $0$ or $1$).

We can get a version with positive $N$, for simplicity, we assume that $N$ is odd.
For this we need to consider the parabolic category $\mathcal{O}$
for the Lusztig form of a quantum group $U_\epsilon(\mathfrak{gl}_n)$, where $\epsilon$ is a primitive $N$th root of
$1$. This is a highest weight category, see [Section 3, Theorem 4.1]\cite{AndMaz}, where the case of the full category $\mathcal{O}$
was considered, the same arguments work in the parabolic case.  Let us define categorical $\sl_2$-actions on $\mathcal{O}$. 
Let $V$ denote the tautological representation of $U_{\epsilon}(\mathfrak{gl}_n)$.
For a $U_{\epsilon}(\mathfrak{gl}_n)$-module $M$ from $\mathcal{O}$, we set $X_M:=R^{21}R: V\otimes M\rightarrow V\otimes M$,
where $R$ is the $R$-matrix. Also set $T_M:=(\sigma\circ R)\otimes \operatorname{id}:V\otimes V\otimes M\rightarrow
V\otimes V\otimes M$, where $\sigma$ is the transposition. The data of $E=V\otimes \bullet, F=V^*\otimes \bullet, X,T$
define the structure of a categorical $\hat{\sl}_N$-action on $\mathcal{O}$, that gives rise to the $N$ categorical
actions of $\sl_2$. 
It is not difficult to see that (i) and (ii) hold. The category does not satisfy (iii$^1$)  but satisfies
(iii$^3$) -- it is defined over $\mathbb{Q}[\epsilon]$.

Another way to get a version with positive $N=p$ is when we consider the category $\Cat$ of rational
representations of $\GL_m(\C)$ with $\C$ being an algebraically closed field of characteristic $p$
and $\ell=1$. See \cite{BK_functors} and \cite{CR} for the description of the categorification.
The categorification satisfies (i) and (ii) as well as (iii$^3$): it is actually defined over
$\mathbb{F}_p$.

Let us explain why the latter holds. Let $U_{\mathbb{F}_p}, U_{\C}$ be the hyperalgebras
of $\gl_n({\mathbb{F}_p}),\gl_n(\C)$ so that $U_{\C}=\C\otimes_{\mathbb{F}_p}U_{\mathbb{F}_p}$
(as associative algebras and also as Hopf algebras). Then a rational representation of $\GL_n(\C)$ is the same
as a finite dimensional $U_{\C}$-module, where the characters of the hyperalgebra of
the Cartan in $\gl_n(\C)$ are integral. The category of finite dimensional $U_{\mathbb{F}_p}$-modules
(with the same integrality condition) is  highest weight. Indeed, the Weyl modules  are defined over $\mathbb{F}_p$.
So (HW1) holds for $U_{\mathbb{F}_p}$. Further, one can construct projectives (over $U_{\C}$)
as extensions of standards as explained, for example, in \cite[Proposition 4.13]{rouqqsch}. This procedure
implies that the projectives are defined over $U_{\mathbb{F}_p}$ and (HW2) is satisfied over $U_{\mathbb{F}_p}$.
Further, the categorical $\sl_2$-action on the category for $U_{\C}$ is defined over $\mathbb{F}_p$ (the operator $X$ is a tensor Casimir and
all its eigenvalues belong to $\mathbb{F}_p$) and so descends the category for $U_{\mathbb{F}_p}$. It is also
clear that the latter action satisfies (i) and (ii).

We can get a version of this construction for $\ell>1$ if we consider parabolic categories $\mathcal{O}$
for $U_{\C}$ (all weights are supposed to be integral and all weight subspaces have to be finite dimensional).

Let us proceed to the case when $\Lambda$ is the poset of multipartitions. Here we can consider
the category $\Cat$ that is the direct sum of categories $\mathcal{O}$ over cyclotomic Cherednik algebras
(all with the same parameters). The categorification itself was defined in \cite{Shan}, while (i) and (ii) were checked in \cite{cryst}.
The claim that one can choose a highest weight order on $\Lambda$ as specified in Subsection \ref{SS_hier_ex}
was essentially established by Griffeth, \cite{Griffeth}, compare with the proof of  \cite[Theorem 1.2]{DG}.
The category $\Cat$ satisfies all three conditions (iii$^1$)-(iii$^3$).

Of special interest for us in this paper will be so called  basic categorifications.
A highest weight $\sl_2$-categorification with poset $\{+,-\}^n$ is said to be {\it basic}.
In characteristic 0, an example is provided by the sum of blocks in the BGG category $\OCat$ for $\gl_n$, as explained
in the introduction. This basic categorification will be called {\it standard} in the sequel.

To the best of our knowledge,  basic categorifications in the positive characteristic $p$ have not appeared explicitly in the literature.
Filtration results from the next section allow to prove  that such categories appear as subquotients of, say,
the category of rational representations of $\operatorname{GL}_n(\C)$.
%Also, it seems, one can realize them
%as Koszul duals of certain categories introduced and studied by Brundan and Stroppel in \cite{BS1,BS2,BS3}.
%Let us explain this in some more detail although we do not need this description in the paper.
%In \cite{BS1}, Brundan and Stroppel introduced certain graded algebras  $K(n-m,m)$, where $0\leqslant m\leqslant n$, via diagram calculus
%and checked that they are quasi-hereditary (meaning that their categories of modules are highest weight).
%These algebras can be defined over an arbitrary field. Then in \cite{BS2} they proved that these algebras
%are Koszul. Moreover, \cite[Theorem 2.1]{CPS} applies and we see that the Koszul dual algebras $K(n-m,m)^!$
%are also quasi-hereditary. So the category $\Cat:=\bigoplus_{m=0}^n K(n-m,m)^!$-$\operatorname{mod}$ is highest
%weight. In characteristic 0, the category $K(n-m,m)$-$\operatorname{mod}$ is isomorphic to
%the principal block of the parabolic category $\mathcal{O}$ for $\gl_n$, where the parabolic
%has blocks of sizes $n-m,m$, see \cite{BS3}. From the parabolic singular duality, \cite{BGS}, it follows that, again, in characteristic
%$0$, the category $\Cat$ is isomorphic to a basic categorification (as a highest weight category).
%In general, it should not be difficult to introduce categorification functors on $\Cat$ that will
%turn it into a basic categorification (with highest weight poset $\{+,-\}^n$).
%\footnote{A discussion on basic
%categorifications in positive characteristic!}

\section{Categorical splitting and family filtration}\label{S_cat_spl}
In this section we will prove two different results. First, we will produce a reduction procedure
that, from a highest weight $\sl_2$-categorification $\Cat$ (with respect to a hierarchy structure on a poset
$\Lambda$) and a family $\Lambda_a$,
will produce isomorphic  categorification structures on the highest weight subquotients of $\Cat$
associated to $\underline{\Lambda}_+^a,\underline{\Lambda}_-^a$. These will be highest weight
categorifications with respect to the hierarchy structures on $\underline{\Lambda}^a$.

From this construction we will deduce that each family is an interval in $\Lambda$ if we consider
$\Lambda$ as a poset with respect to the coarsest possible ordering compatible with the highest
weight structure on $\Lambda$. Recall that ``$\Lambda_a$ is an interval'' means:  if $\lambda_1,\lambda_2\in \Lambda_a$ and $\mu$ lies between $\lambda_1,\lambda_2$ in that ordering, then $\mu\in \Lambda_a$.   In particular,  there is a filtration on $\Cat$ (compatible with both
categorification and highest weight structures, as explained in the introduction) whose successive quotients
are basic $\sl_2$-categorifications. We remark that in all examples it is possible to see this just from
the combinatorics (this is literally so for our poset structure on the multipartitions;
in other cases we need to take coarser orderings).

\subsection{Categorical splitting: a setting}
Recall that $\Lambda$ stands for a poset equipped with a hierarchy structure.
Fix a family $\Lambda_a$.
Recall that we have decomposed the poset $\Lambda$ into the union of intervals (we suppress the superscript ``$a$'')
\begin{equation}\label{eq:comb_splitting}\Lambda=\Lambda_{<}\sqcup \underline{\Lambda}_-\sqcup \underline{\Lambda}_+\sqcup \Lambda_{>},\end{equation}
where the terms are written in a non-decreasing order and the posets $\underline{\Lambda}_?$ are isomorphic
to a single poset $\underline{\Lambda}$. We set $\Lambda_=:=\underline{\Lambda}_-\sqcup \underline{\Lambda}_+,
\Lambda_{\leqslant}:=\Lambda_{<}\sqcup \Lambda_=$. Consider the Serre subcategories $\Cat_<, \Cat_{\leqslant}$
spanned by the simples $L(\lambda)$ with $\lambda\in \Lambda_{<},\Lambda_{\leqslant}$. These are highest
weight subcategories. Since the class of each simple in $\Cat$ is a weight vector, we see that $\Cat_{<},\Cat_{\leqslant},
\Cat_=$ inherit the weight decomposition from $\Cat$. Moreover, the properties of the decomposition (\ref{eq:comb_splitting})
imply that these subcategories are closed with respect to $E,F$. We claim that $\Cat_<,\Cat_{\leqslant}$
as well as the quotient $\Cat_=:=\Cat_{\leqslant}/\Cat_<$ inherit  categorical $\sl_2$-actions from
$\Cat$. Indeed, let  $\Cat'\subset \Cat$ be a Serre subcategory and let $\mathcal{E},\mathcal{F}$ be
functors preserving the subcategory, then we have maps $\operatorname{Hom}_{\Cat}(\mathcal{E},\mathcal{F})
\rightarrow \operatorname{Hom}_{\Cat'}(\mathcal{E},\mathcal{F}), \operatorname{Hom}_{\Cat/\Cat'}(\mathcal{E},\mathcal{F})$
that are compatible with compositions. This implies that the functors $E,F$ together with the transformations
$X,T$ induced from $\Cat$ define categorical $\sl_2$-actions on $\Cat_<,\Cat_{\leqslant}, \Cat_=$.

Besides, $\Cat_=$ is a highest weight category, see Subsection \ref{SS_hw_reminder}.
It is straightforward to see that $(\Cat_=,\Lambda_=)$ is a highest weight $\sl_2$-categorification
with respect to the hierarchy structure on $\Lambda_=$ restricted from $\Lambda$. Below we will therefore assume that $\Lambda_==\Lambda$
and so $\Cat_==\Cat$.

So the category $\Cat$ has a Serre subcategory $\Cat_-$ corresponding to $\underline{\Lambda}_-$.
Form the quotient $\Cat_+=\Cat/\Cat_-$. These are highest weight categories with posets $\underline{\Lambda}_-,
\underline{\Lambda}_+$. However, let us notice that $\Cat_-$ is not a sub-categorification. Indeed, $E$
preserves $\Cat_-$ but $F$ does not, roughly speaking, $F$ can switch the last $-$ into a $+$.

Our goal in this section will be to introduce categorical $\sl_2$-actions on $\Cat_+,\Cat_-$ that turn
them into highest weight categorifications with respect to the hierarchy structure on $\underline{\Lambda}$
and, moreover, to show that $\Cat_+,\Cat_-$ are isomorphic as highest weight categorifications.

We proceed by defining certain functors that will be shown to be required equivalences.

\subsection{Functors $\Fun,\Eun$}\label{SS_Eun}
Let $\iota,\pi$ denote the embedding $\Cat_-\hookrightarrow \Cat$ and the quotient functor $\Cat\twoheadrightarrow\Cat_+$, respectively.
For $\lambda\in \underline{\Lambda}$ we define elements $\lambda?\in \underline{\Lambda}_?$, where $?=+,-$,
using the natural bijections $\underline{\Lambda}\xrightarrow{\sim}\underline{\Lambda}_?$.

Define a functor $\Fun:\Cat_-\rightarrow \Cat_+$ by $\Fun:=\pi\circ F\circ \iota$. Let us list some simple properties
of $\Fun$.

\begin{Lem}\label{Lem:Fun}
The following assertions hold:
\begin{enumerate}
\item The functor $\Fun$ is exact.
\item $\Fun(\Delta_-(\lambda))=\Delta_+(\lambda), \Fun(\nabla_-(\lambda))=\nabla_+(\lambda)$. Here
for $\lambda\in \underline{\Lambda}$ by $\Delta_?(\lambda),\nabla_?(\lambda)$ we denote the standard
and costandard objects in $\Cat_?$ corresponding to $\lambda$ (with $?=\pm$).
\item Under the standard identifications of the rational $K_0$-groups $[\Cat_{?}]$ with
$\Q^{\underline{\Lambda}}$, the functor $\Fun$ induces the identity map on the $K_0$-groups.
\item $\Fun(L_-(\lambda))=L_+(\lambda)$ for any $\lambda\in \underline{\Lambda}$.
\end{enumerate}
\end{Lem}
\begin{proof}
(1) follows because $\Fun$ is the composition of three exact functors. To prove $\Fun(\Delta_-(\lambda))=\Delta_+(\lambda)$
recall the standard filtration on $F\Delta(\lambda-)$ mentioned in Subsection \ref{SS_hw_def}.
Let us notice that in the standard filtration of $F\Delta(\lambda-)$ the only successive quotient that does not lie in $\Cat_-$
is the subobject $\Delta(\lambda+)$. Since $\pi(\Delta(\lambda+))=\Delta_+(\lambda)$, we are done. Applying the same
argument to $\Cat^{opp}$, we prove that $\Fun(\nabla_-(\lambda))=\nabla_+(\lambda)$. This completes the proof of (2).
(3) easily follows. To prove (4) we notice that $\Fun(L_-(\lambda))\neq 0$ because of (3). Recall
that $L_-(\lambda)$ is the image of any nonzero morphism $\Delta_-(\lambda)\rightarrow \nabla_-(\lambda)$. Since $\Fun$ is exact,
it maps $L_-(\lambda)$ to the image of a morphism $\Delta_+(\lambda)=\Fun(\Delta_-(\lambda))\rightarrow \Fun(\nabla_-(\lambda))=\nabla_+(\lambda)$.
Since $\Fun(L_-(\lambda))\neq 0$, we see that $\Fun(L_-(\lambda))=L_+(\lambda)$.
\end{proof}

Now let us define a functor $\Eun:\Cat_+\rightarrow \Cat_-$.
Let $\pi^!,\iota^!$ denote the left adjoints of the functors $\pi,\iota$, see Subsection \ref{SS_hw_reminder}.
We set $\Eun:=\iota^!\circ E\circ \pi^!$. The following lemma describes some basic properties of $\Eun$.

\begin{Lem}\label{Lem:Eun}
The following assertions hold:
\begin{enumerate}
\item The functor $\Eun$ is left adjoint to $\Fun$.
\item Moreover, $\Eun$ maps the indecomposable
projective $P_+(\lambda)$ to the indecomposable projective $P_-(\lambda)$.
\item The natural morphism $\Eun\Fun N\twoheadrightarrow N$ is surjective for
any $N\in \Cat_-$.
\end{enumerate}
\end{Lem}
\begin{proof}
(1) follows directly from the constructions of $\mathcal{E},\mathcal{F}$.

Let us prove (2). Since $\Fun$ is exact, $\Eun$  maps projectives to
projectives. Further, $\pi^!(P_+(\lambda))=P(\lambda+)$ by the definition of $\pi^!$.
We have $EP(\lambda+)\twoheadrightarrow E\Delta(\lambda+)\twoheadrightarrow \Delta(\lambda-)$ and
$\Delta(\lambda-)$ lies in $\Cat_-$. So $\Eun P_+(\lambda)\twoheadrightarrow \Delta_-(\lambda)$.
Hence $\Eun P_+(\lambda)$ contains $P_-(\lambda)$ as a direct summand. To prove that $\Eun P_+(\lambda)=P_-(\lambda)$
we need to prove that $\dim \Hom(\Eun P_+(\lambda), L_-(\mu))=\delta_{\lambda,\mu}$. But assertion (4)
of Lemma \ref{Lem:Fun} says $\Fun L_+(\mu) =L_-(\mu)$ and so
$$\Hom(\Eun P_+(\lambda), L_-(\mu))=\Hom(P_+(\lambda), \Fun L_-(\mu))=\Hom(P_+(\lambda), L_+(\mu))$$
that completes the proof of (2).

Let us prove (3). Consider the exact sequence $\Eun\Fun N\rightarrow N\rightarrow K\rightarrow 0$. Apply $\Fun$
to this sequence to get an exact sequence $\Fun\Eun\Fun N\rightarrow \Fun N\rightarrow \Fun K\rightarrow 0$.
But the first arrow is surjective, thanks to the adjointness. So $\Fun K=0$. Since $\Fun$ is exact and
induces a bijection on the Grothendieck groups, we see that $K=0$.
\end{proof}

Our goal is to prove that $\Fun,\Eun$ are quasi-inverse equivalences. This is achieved in the next two lemmas.
The first one describes the behavior of $\Eun$ on standardly filtered objects.

\begin{Lem}\label{Lem:Eun_stand}
We have the following
\begin{enumerate}
\item $\Eun(\Delta_+(\lambda))=\Delta_-(\lambda)$ for any $\lambda\in \underline{\Lambda}$.
%\item If $0\rightarrow N_1\rightarrow N_2\rightarrow N_3\rightarrow 0$ is an exact sequence in
%$\Cat_+^\Delta$, then the sequence $0\rightarrow \Eun(N_1)\rightarrow \Eun(N_2)\rightarrow \Eun(N_3)\rightarrow 0$
%is exact. In particular, $\Eun$ maps $\Cat_+^\Delta$ to $\Cat_-^\Delta$.
\item $\Eun$ is exact on standardly filtered objects.
\item For $N\in \Cat_-^\Delta$ we have $\Eun\Fun(N)\xrightarrow{\sim}N$.
\end{enumerate}
\end{Lem}
\begin{proof}
Let us prove (1). By Lemma \ref{Lem:hw_quot}, $\pi^!\pi(N)\rightarrow N$
is an isomorphism for any object $N\in \Cat$ admitting a filtration with quotients
of the form $\Delta(\mu+)$. In particular,  $\pi^!\Delta_+(\lambda)=\Delta(\lambda+)$.

Apply $E$ to $\Delta(\lambda+)$. The top quotient of the standard filtration is $\Delta(\lambda-)$
and all the other subquotients are $\Delta(\mu)$ with $\mu\in \underline{\Lambda}_+$.
So $\iota^! E\Delta(\lambda+)=\Delta(\lambda-)$ and therefore $\Eun\Delta_+(\lambda)=\Delta_-(\lambda)$.

To prove (2) we notice that $\pi^!,E,\iota^!$ map standardly filtered objects to standardly filtered
ones and are exact on standardly filtered objects.

(3) follows from (1) and (2): we use $\mathcal{EF}\Delta_-(\lambda)\xrightarrow{\sim}\Delta_-(\lambda)$, the fact
that $\mathcal{EF}$ is exact on standardly filtered objects, and apply the 5-lemma.
\end{proof}

The following lemma finally implies that $\Fun$ is an equivalence.

\begin{Lem}\label{Lem:Fun_proj}
We have $\Fun(P_-(\lambda))=P_+(\lambda)$ and $\Fun$ is fully faithful on $\Cat$-$\operatorname{proj}$.
\end{Lem}
Here, as usual, $\Cat$-$\operatorname{proj}$ denotes the full subcategory of $\Cat$ consisting of
the projective modules.

\begin{proof}
We have an isomorphism $$\sigma:\Hom(P_+(\lambda), \Fun P_-(\lambda))\xrightarrow{\sim}\Hom(\Eun P_+(\lambda), P_-(\lambda))\xrightarrow{\sim}\Hom(P_-(\lambda), P_-(\lambda)).$$

We claim that $\varphi:=\sigma^{-1}(\operatorname{id})$ is an isomorphism.
To check this recall that $\sigma$ is obtained as follows: for $\varphi\in \Hom(P_+(\lambda), \Fun P_-(\lambda))$
we have $\sigma(\varphi)=\eta\circ \Eun\varphi\circ \nu$, where $\eta$ is a natural morphism $\Eun\Fun P_-(\lambda)\rightarrow
P_-(\lambda)$ that was shown to be an isomorphism in Lemma \ref{Lem:Eun}, and $\nu$ is an isomorphism $P_-(\lambda)\xrightarrow{\sim}\Eun P_+(\lambda)$, see Lemma \ref{Lem:Eun_stand}.

Let us show that $\varphi$ is surjective. Consider the exact sequence $P_+(\lambda)\xrightarrow{\varphi} \Fun P_-(\lambda)\rightarrow N\rightarrow 0$ and apply $\Eun$ to it. We get an exact sequence $P_-(\lambda)\rightarrow P_-(\lambda)\rightarrow \Eun(N)\rightarrow 0$.
But the first arrow is nothing else but $\sigma(\varphi)$, i.e., is the identity. So $\Eun(N)=0$. Let us check that this implies
$N=0$. Indeed, $0=\Hom(\Eun(N),L_-(\mu))=\Hom(N, \Fun L_-(\mu))$ but the latter is $\Hom(N,L_+(\mu))$ by the
last assertion of Lemma \ref{Lem:Fun}. It follows that $N$ has no head and hence is 0. So we have proved that $\varphi$
is surjective.

To prove that $\varphi$ is an isomorphism it remains to show that $[P_+(\lambda)]=[\Fun(P_-(\lambda))]$.
Since $\Fun$ is exact and maps $\Delta_-(\lambda)$ to $\Delta_+(\lambda)$, under our identification of the
$K_0$-groups, the class of $\Fun(P_-(\lambda))$ coincides with that of $P_-(\lambda)$. So it remains to show that, for all $\lambda,\mu\in \underline{\Lambda}$, the multiciplities $(P_+(\lambda):\Delta_+(\mu)),(P_-(\lambda):\Delta_-(\mu))$ are equal.
But thanks to the BGG reciprocity, this is equivalent to checking $(\nabla_+(\mu):L_+(\lambda))=(\nabla_-(\mu):L_-(\lambda))$.
The latter follows from the exactness of $\Fun$ and  the isomorphisms $\Fun(\nabla_-(\mu))=\nabla_+(\mu), \Fun(L_-(\lambda))=L_+(\lambda)$ that were established in Lemma \ref{Lem:Fun}. The proof that $\varphi$ is an isomorphism is complete.

The claim that $\Fun$ is fully faithful on projectives follows now from
$$\Hom(\Fun P_-(\lambda), \Fun P_-(\mu))=\Hom(\Eun \Fun P_-(\lambda), P_-(\mu))=\Hom(P_-(\lambda), P_-(\mu)).$$
\end{proof}

\subsection{Categorifications on $\Cat_\pm$}\label{SS:cat_pm}
The goal of this subsection is to prove the following proposition.

\begin{Prop}\label{Prop:reduced_cat}
The functors $\underline{E}:=E, \underline{F}:=\Eun \pi F^{(2)}\iota$ define a structure of
a highest weight $\sl_2$-categorification (with respect to $\underline{\Lambda}$) on $\Cat_-$.
\end{Prop}

Let us point out that $E$ preserves $\Cat_-$ because it is exact and preserves $\Cat_-^\Delta$.
The  most non-trivial part of the proof is to show that $\underline{F}$ is isomorphic both to the left
and to the right adjoint of $\underline{E}$.

The functor $\underline{E}$ does have both left and right adjoint functors $\underline{E}^!$ and $\underline{E}^*$.
They are constructed as follows: $\underline{E}^!=\iota^! F \iota, \underline{E}^*=\iota^* F \iota$, where $\iota^*$
is the right adjoint to $\iota$ (sending an object $N$ to its maximal submodule belonging to $\Cat_-$).
We will prove that $\underline{E}^!$ is isomorphic to $\underline{F}$. The analogous statement
for $\underline{E}^*$ is proved by passing to $\Cat^{opp}$. We will actually prove that $\Fun\underline{E}^!=\pi F\iota\iota^! F\iota$
is isomorphic to $\Fun\underline{F}=\pi F^{(2)}\iota$.

%Our first step in the proof of the  isomorphism will be to establish some properties of the functors $\pi^!\pi,\iota\iota^!$.
%
%\begin{Lem}\label{Lem:adj_fun_morph}
%The following assertions hold:
%\begin{enumerate}
%\item There is an exact sequence of functors $\pi^!\pi\rightarrow \operatorname{id}\rightarrow \iota\iota^!\rightarrow 0$.
%\item The functors $\pi^!\pi, \iota\iota^!$ are exact endofunctors of $\Cat^\Delta$.
%\item The sequence $0\rightarrow \pi^!\pi \rightarrow \operatorname{id}\rightarrow \iota\iota^!\rightarrow 0$ is exact on
%$\Cat^\Delta$.
%\end{enumerate}
%\end{Lem}
%\begin{proof}
%The first part is standard. Consider a standardly filtered object $N\in \Cat$. Recall that $\underline{\Lambda}_-\not>\underline{\Lambda}_+$. %By the axioms of a highest weight category, we can find a filtration on $N$ that has a filtration component $N_1\subset N$  such that
%\begin{itemize}
%\item
%the successive subquotients of $N_1$ are $\Delta(\mu)$ with $\mu\in \underline{\Lambda}_+$,
%\item   successive subquotients of the quotient $N_2:=N/N_1$ are $\Delta(\nu)$'s with $\nu\in \underline{\Lambda}_-$.
%\end{itemize}
%As we have mentioned in the proof of Lemma \ref{Lem:Eun_stand}, $\pi^!\pi(N_1)=N_1$. Since $\pi(N)=\pi(N_1)$,
%we see that $N_1=\pi^!\pi(N)$ (in particular, this shows that $N_1$
%does not depend on the choice of a filtration). Clearly, $\iota\iota^!(N)=N_2$. Assertions (2) and (3) follow.
%\end{proof}

Composing the exact sequence $\pi^!\pi\rightarrow \operatorname{id}\rightarrow \iota\iota^!\rightarrow 0$,
see Lemma \ref{Lem:hw_quot}, (iii),  with $\pi F$ on the left and $F\iota$ on the right
we have an exact sequence of functors $\Cat_-\rightarrow \Cat_+$
\begin{equation}\label{eq:fun_seq}
\pi F\pi^!\pi F \iota\rightarrow \pi F^2\iota\rightarrow \pi F\iota\iota^! F\iota\rightarrow 0.
\end{equation}
The left functor morphism becomes injective on $\Cat_-^\Delta$.

The proof of the following lemma is straightforward.

\begin{Lem}\label{Lem:K_group}
In $[\Cat_+]\cong \Q^{\underline{\Lambda}}$ we have $$[\pi F\pi^!\pi F \iota\left(\Delta_-(\lambda)\right)]=[\pi F\iota\iota^! F\iota\left(\Delta_-(\lambda)\right)]=\sum_\mu [\Delta_+(\mu)],$$
where the summation is taken over all $\mu\in \underline{\Lambda}$ that are obtained from $\lambda$ by replacing
one $-$ with a +.
\end{Lem}

Recall that we have a functor isomorphism $\C^2\otimes_{\C}F^{(2)}\cong F^2$. We claim that one can choose
an embedding $\C\hookrightarrow \C^2$ such that the corresponding composition $\pi F^{(2)}\iota\hookrightarrow
\pi F^2\iota\rightarrow \pi F\iota\iota^! F\iota$ is an isomorphism.

For this we will show that the set of all embeddings $\C\hookrightarrow \C^2$ such that the morphism \begin{equation}\label{eq:req_iso}\pi F^{(2)}\iota(\Delta_-(\lambda))\rightarrow \pi F\iota\iota^! F\iota(\Delta_-(\lambda))\end{equation}
is iso is open in $\mathbb{P}^1$.
The morphism is an iso if and only if the map $$\C\cong\Hom(\pi F\iota\iota^! F\iota(\Delta_-(\lambda)), \nabla(\mu))\rightarrow
\Hom(\pi F^{(2)}\iota(\Delta_-(\lambda)), \nabla(\mu))\cong \C$$
is non-zero for $\mu$ as in Lemma \ref{Lem:K_group} (for all other $\mu$'s the spaces
involved are zero). But this map is a composition of an embedding
\begin{equation}\label{eq:embedding}\Hom(\pi F\iota\iota^! F\iota(\Delta_-(\lambda)), \nabla(\mu))\rightarrow
\Hom(\pi F^{2}\iota(\Delta_-(\lambda)), \nabla(\mu))\cong \C^2\end{equation}
and the projection
\begin{equation}\label{eq:projection}
\Hom(\pi F^{2}\iota(\Delta_-(\lambda)), \nabla(\mu))\rightarrow
\Hom(\pi F^{(2)}\iota(\Delta_-(\lambda)), \nabla(\mu)).\end{equation}
The latter  is the dual of our embedding $\C\hookrightarrow \C^2$ and therefore  the claim in the beginning of the
paragraph holds.

Moreover, if (iii$^2$) holds, then
(\ref{eq:req_iso}) can be made an isomorphism for all $\lambda$ (recall,  we assume
that $\Lambda$ is countable). It follows that we have an isomorphism of right exact functors $\underline{E}^!=\pi F\iota\iota^! F\iota,
\pi F^{(2)}\iota$ on $\Cat$-$\operatorname{proj}$ (because any projective is $\Delta$-filtered) and hence on $\Cat$.

If (iii$^1$) holds for $\Cat$, then, similarly to the previous paragraph, we see that
the functors $\pi F\iota\iota^! F\iota$ and $\pi F^{(2)}\iota$ are isomorphic on $\Cat$-$\operatorname{proj}$
blockwise and hence are isomorphic.

Finally, let us suppose that (iii$^3$) holds. Then the embedding
(\ref{eq:embedding}) is defined over $\C'$ for all $\lambda,\mu$. It follows that we can take a finite extension
$\C''$ of $\C'$ and a projection (\ref{eq:projection}) defined over $\C''$ such that composition
of (\ref{eq:embedding}) and (\ref{eq:projection}) is an isomorphism. This again implies that the functors
$\pi F^{(2)} \iota$ and $\underline{E}^!$ are isomorphic.

%If $\underline{\Lambda}$
%was finite, this would complete the proof. In general, $\underline{\Lambda}$ is infinite. So our argument
%works  if all blocks in $\Cat_-$ are finite because it is enough to establish an isomorphism blockwise
%or if the base field is uncountable.

%A consequence of the preceding discussion is that there is a functor morphism $\pi F^{(2)}\iota\rightarrow \underline{E}^!$
%that is an isomorphism on all standardly filtered objects. Hence it is an isomorphism on all projectives and, since
%both functors are right exact, is an isomorphism everywhere.

Checking that the functors $\underline{E},\underline{F}$ form an $\sl_2$-categorification
is now easy (there are natural transformations $X$ of $\underline{E}$ and $T$ of $\underline{E}^2$
induced from the analogous transformations of $E,E^2$). Let us check that this categorification
is  highest weight with respect to the
hierarchy structure on $\underline{\Lambda}$. (i) is clear for $\underline{E}$, while  for $\underline{F}$
it follows directly from the construction.  (ii) is again clear for $\underline{E}$ and for
$\underline{F}$ it follows from  Lemma \ref{Lem:K_group}. Let us remark on the conditions (iii$^?$). Obviously,
(iii$^2$) is preserved. It is also clear that (iii$^1$) is preserved. As for (iii$^3$), the category
$\Cat^-$ is defined over $\C''$ from the construction.

A categorification on $\Cat_+$ is obtained by transferring the categorification
on $\Cat_-$ via the equivalence $\mathcal{F}$.

\begin{Rem}\label{Rem:splitting_duality}
The splitting construction can be adapted  to the dual hierarchy setting, as well.
Given a highest weight $\sl_2$-categorification $\Cat$ with respect to a hierarchy structure on $\Lambda$ we can apply the splitting construction to $\bar{\Cat}$. We get the subquotient $\bar{\Cat}_=$ of $\Cat$ with poset $\bar{\Lambda}_=$, whose categorification is inherited from $\Cat$ in a naive way. Then we have an extension $0\rightarrow \bar{\Cat}_+\rightarrow \bar{\Cat}_=\rightarrow \bar{\Cat}_-\rightarrow 0$. The categories $\bar{\Cat}_+, \bar{\Cat}_-$ come equipped with equivalent highest weight $\sl_2$-categorification structures (with respect to the dual hierarchy structure on $\underline{\bar{\Cat}}$).
\end{Rem}

\begin{Rem}\label{Rem:F_stand}
We still assume that $\Lambda=\Lambda^a_=$.
The functor $\pi^!$ defines an equivalence of $\Cat_+^\Delta$ with the full subcategory in
$\Cat^\Delta$ of all objects that admit a filtration with subquotients  $\Delta(\lambda), \lambda\in \underline{\Lambda}^a_+$.
Then on $\Cat_+^\Delta$ the functor $\underline{F}$ coincides with $F$. Indeed, we can describe $\underline{E}$ on $\Cat_+^\Delta$
as the composition of $E$ with taking the maximal submodule lying in $\Cat_+^\Delta$. From this description it is
clear that the restriction of $F$ to $\Cat^\Delta_+$ is left adjoint to $\underline{E}$ and hence
is isomorphic to $\underline{F}$.
\end{Rem}

\subsection{Filtration}\label{SS_family_filtr}
Our goal here is to explain a technique that will reduce the study of some questions about
$\Cat$ to the case when $\Lambda$ is a single family $\Lambda_a$.
Let $\prec$ be the coarsest possible ordering on $\Lambda$ making $\Cat$ into
a highest weight category, see Remark \ref{Rem:ordering}.
For disjoint subsets $\Lambda_1,\Lambda_2$ of $\Lambda$ we write $\Lambda_1\prec\Lambda_2$ if
$\lambda_1\prec\lambda_2$ for some $\lambda_1\in \Lambda_1,\lambda_2\in \Lambda_2$.
The main result of this subsection is the following proposition.

\begin{Prop}\label{Prop:filtration}
The transitive closure of $\prec$ is  a partial order on the set of families.
\end{Prop}

So we can choose a filtration on the poset $\Lambda$ by ideals such that the subsequent quotients
are families. This filtration (and the corresponding filtration on $\Cat$) will be called {\it family
filtrations}.

\begin{proof}
What we need to prove is that there are no different indexes $a_1,\ldots,a_k$ such that $\Lambda_{a_1}\prec\Lambda_{a_2}\prec\ldots \prec \Lambda_{a_k}\prec \Lambda_{a_1}$.

We will need some terminology. We say that elements
$\lambda_1,\lambda_2,\ldots, \lambda_n\in \Lambda$ form a {\it chain} if $\Ext^1(\Delta(\lambda_i), L(\lambda_{i+1}))\neq 0$
or $\Hom(\Delta(\lambda_i),\Delta(\lambda_{i+1}))\neq 0$
for all $i=1,\ldots,n-1$. We claim that $\lambda\prec\mu$ if and only if $\lambda$ and $\mu$
can be connected by a chain, i.e., there is a chain $\lambda_1,\ldots,\lambda_n$ with $\lambda_1=\lambda,\lambda_n=\mu$.
Indeed, it is enough to prove this claim when $\lambda\prec \mu$ but there are no elements between $\lambda$ and $\mu$.
The condition that $\Ext^1(\Delta(\lambda), L(\mu))=0$ now means that $\Delta(\lambda)$ does not occur in $P(\mu)$.
So, for a highest weight ordering on $\Cat$, we can take the coarser ordering where $\lambda$ is not comparable with $\mu$ and the other relations are the same
as in $\prec$.

We say that a chain is {\it tight} if it cannot be refined, i.e., there are no $i$ and $\nu$ with $\lambda_i\prec \nu\prec \lambda_{i+1}$.

We are going to prove our claim in the beginning of the proof using the
artinian induction on the hierarchy, such an induction works thanks to
(H2). To establish the induction step assume the contrary. Let us consider the decomposition $\Lambda=\Lambda_{<}\sqcup \underline{\Lambda}_-\sqcup\underline{\Lambda}_+\sqcup \Lambda_{>}$ corresponding to $\Lambda_{a_1}$. The decompositions for $\Lambda_{a_i}$ are the same for all $i$, otherwise some family will be strictly less than the others. We may assume that $\Lambda=\Lambda_=$.

To prove the induction step, it is enough to check that $\Lambda_a\prec \Lambda_b$ implies the existence of $a_1',\ldots,a_i'$
with $a_1'=a, a_k'=b$ and $\underline{\Lambda}_{a_1'}\prec \underline{\Lambda}_{a_2'}\prec\ldots\prec\underline{\Lambda}_{a_k'}$
(in $\underline{\Lambda}$). So pick $\lambda\in \Lambda_a, \mu\in \Lambda_b$ that can be connected by a chain. If both lie in $\underline{\Lambda}_+$ or both lie in $\underline{\Lambda}_-$, then we can use  the inductive assumption
(with the categorification $\Cat_+$ or $\Cat_-$). We only need to consider the case when
$\lambda\in \underline{\Lambda}_-$, while $\mu\in \underline{\Lambda}_+$.

Let us include $\lambda$ and $\mu$ into a tight chain $\lambda=\lambda_1,\lambda_2,\ldots,\lambda_n=\mu$.
Let $m$ be such that $\lambda_m\in \underline{\Lambda}_-, \lambda_{m+1}\in \underline{\Lambda}_+$.
If $\lambda_m,\lambda_{m+1}$ lie in the same family, say $\Lambda_c$, then we can just take $a_1'=a, a_2'=c,a_3'=b$.
So assume that they lie in  different families: $\lambda_{m}\in\Lambda_c$ and $\lambda_{m+1}\in \Lambda_d$.
Let $\tilde{\lambda}_m$ denote the element obtained from $\lambda_m$ by replacing the right-most $-$ with a $+$.
What we need to show is that there is an element $\lambda_{m+1}'\in \underline{\Lambda}_+\cap \Lambda_d$
with $\tilde{\lambda}_m\prec \lambda_{m+1}'$.
%It is sufficient to prove that $\tilde{\lambda}_m\prec \lambda_{m+1}$: then we have a required ordered collection of
%families in $\underline{\Lambda}$.

The top quotient of the standard filtration on $E\Delta(\tilde{\lambda}_m)$ is $\Delta(\lambda_m)$
and all other successive subquotients are  different from $\Delta(\lambda_{m+1})$ and have labels
in $\underline{\Lambda}_+$. So if $\Hom(\Delta(\lambda_m),\Delta(\lambda_{m+1}))\neq 0$, then
$0\neq\Hom(E\Delta(\tilde{\lambda}_m),\Delta(\lambda_{m+1}))=\Hom(\Delta(\tilde{\lambda}_m), F\Delta(\lambda_{m+1}))$.
Since the latter space is nonzero, we see that $\tilde{\lambda}_{m+1}\prec \lambda_{m+1}'$, where
$\Delta(\lambda_{m+1}')$ is one of the standard filtration quotients of $F\Delta(\lambda_{m+1})$.
But $\lambda_{m+1}'$ is in the same family as $\lambda_{m+1}$ and so we are done.

Now it remains to consider the case when $\Ext^1(\Delta(\lambda_m), L(\lambda_{m+1}))\neq 0$.
Suppose, first, that  $\Ext^1(\Delta(\tilde{\lambda}_m), FL(\lambda_{m+1}))=0$. Then $\Ext^1(E\Delta(\tilde{\lambda}_m), L(\lambda_{m+1}))= 0$. Let $M$ denote the kernel of the epimorphism $E\Delta(\tilde{\lambda}_m)\twoheadrightarrow \Delta(\lambda_m)$. Then we have
$\Hom(M, L(\lambda_{m+1}))\twoheadrightarrow \Ext^1(\Delta(\lambda_m),L(\lambda_{m+1}))$ meaning, in particular,
that $\Hom(M, L(\lambda_{m+1}))\neq 0$. In other words, $L(\lambda_{m+1})$ is in the head of $M$. Since $\operatorname{head}$
is a right exact functor and $M$ admits a filtration whose successive quotients are $\Delta(\lambda')$ with $\lambda'\in \Lambda_c$,
we get a contradiction with $\lambda_{m+1}\in \Lambda_d\neq \Lambda_c$.

So $\Ext^1(\Delta(\tilde{\lambda}_m), FL(\lambda_{m+1}))\neq 0$.
It follows that there is a simple constituent $L(\tilde{\lambda}_{m+1})$ of $FL(\lambda_{m+1})$ such that $\Ext^1(\Delta(\tilde{\lambda}_m), L(\tilde{\lambda}_{m+1}))\neq 0$. Therefore $\tilde{\lambda}_{m+1}\in \underline{\Lambda}_+$. Also we have a surjection $F\Delta(\lambda_{m+1})\twoheadrightarrow FL(\lambda_{m+1})$. This means, in particular, that there is $\lambda_{m+1}'\in \Lambda_d\cap \underline{\Lambda}_+$ (the label of a standard object in the standard filtration of $F\Delta(\lambda_{m+1})$) such that
$L(\tilde{\lambda}_{m+1})$ appears in the composition series of $\Delta(\lambda'_{m+1})$.
Therefore $\tilde{\lambda}_{m}\prec \tilde{\lambda}_{m+1}\prec \lambda_{m+1}'$ and so we are done.
\end{proof}

\begin{Rem}\label{Rem:label_order}
In fact, in all our examples one can choose an ordering, where the family filtration
is visible from the combinatorics. For example, the orderings on the poset $\mathcal{P}_\ell$
in Subsection \ref{SS_hier_ex} can be seen to have this property. It is still useful to have
a hierarchy structure defined as it was. It will be used in a subsequent paper, \cite{Cher_mult},
to define categorical actions on certain truncations of affine parabolic categories $\mathcal{O}$.
\end{Rem}

\section{Projective resolutions of standards in basic categorifications}\label{S_proj_resol}
\subsection{Main result}
Let $\Cat$ be a basic highest weight $\sl_2$-categorification with poset $\{+,-\}^n$
(in the sequel, we will call such $\Cat$ a {\it basic categorification of size $n$}).
The goal of this section is to determine a minimal projective resolution of a standard
$\Delta(t)$, equivalently, to compute $\Ext^i(\Delta(t),L(s))$ for all $i\geqslant 0,t,s\in \{+,-\}^n$.

Let us introduce the notion of a {\it division}  of $t$. By definition, a division $D$ consists of two subsets
$I_+,I_-\subset \{1,\ldots,n\}$ of {\it fixed} positions, and pairs $p^1,\ldots,p^k\subset \{1,\ldots,n\}$
subject to the following conditions:
\begin{itemize}
\item[(D1)] $\{1,\ldots,n\}=I_+\sqcup I_-\sqcup \bigsqcup_{i=1}^k p^i$.
\item[(D2)] For any $i_+\in I_+, i_-\in I_-$ we have $i_+<i_-$.
\item[(D3)] If $p^i=\{j,j'\}$, then exactly one of $t_j, t_{j'}$ is a $+$
(and the other is a $-$).
\item[(D4)] Let $p^i=\{j,j'\}$ with $j<j'$. Then on the interval $[j,j']$ there are no elements
of $I_\pm$. Moreover, if $p^{i_1}=\{j_1,j_1'\}$ is another pair, and $j_1\in [j,j']$, then $j_1'\in [j,j']$.
\end{itemize}

Graphically, a division is represented by a cup diagram of \cite{BS1}.

A pair $p^i=\{j,j'\}, j<j',$ is said to be {\it switchable} if $t_j=+, t_{j'}=-$. Let $s(D)$ be the total number
of switchable pairs in $D$. For $t\in \Lambda$ and its division $D$
we define $t^D\in \Lambda$ by switching +'s and $-$'s in all switchable pairs.
For example, consider $t=+-+--$. Then $I_+=\{1\}, I_-=\{2,5\}, p^1=\{3,4\}$ form a division $D_1$.
The pair $p^1$ is switchable, and $t^{D_1}=+--+-$. Another division $D_2$ of $t$ is, say, $I_1=\{1\},
I_2=\{4,5\}, p^1=\{2,3\}$. In this case, $t^{D_2}=t$.

Here is the main result of this section.

\begin{Thm}\label{Thm_proj_resol}
Let $P_\bullet=\ldots\rightarrow P_{i+1}\rightarrow P_i\rightarrow\ldots\rightarrow P_0$ be a minimal projective
resolution of $\Delta(t)$. Then, for any $i$, we have $P_i=\bigoplus_D P(t^D)$, where the sum is taken over
all divisions $D$ of $t$ with $s(D)=i$. In other words, $\Ext^i(\Delta(t), L(s))=1$ if and only if $s=t^D$
for  a division $D$ with $s(D)=i$, otherwise the ext vanishes.
\end{Thm}

In particular, this theorem implies the character formulas for the projectives (and hence for simples) in $\Cat$.

We remark that for given $s$ and $t$ the equality $s=t^D$ holds for at most one $D$ that can be determined as follows.
For $s\in \{+,-\}^n$ we can define its reduced form using the following procedure. On each step we take two
indices $a,b$ such that $s_a=-, s_b=+, s_{a+1}=\ldots=s_{b-1}=0$ and make $s_a,s_b$ equal $0$ (initially there are
no 0's). We finish with an $n$-tuple of $+,-,0$, where no $+$ appears to the right of a $-$.
Of course, this is the standard bracket cancelation recipe,  with a $-$ being a ``(", and a + being a ``)".
For $?=+,-$ let $I_?(s)$ denote the  set of all positions, where we have a $?$. The sets $I_+(s),I_-(s)$ are referred to as the {\it reduced form} of $s$.

The sets $I_+,I_-$ in the division $D$ constitute  the reduced form of $t^D$
and so are uniquely recovered from $t^D$. Any pair $p^i$ is located either between two consecutive elements of $I_+\sqcup I_-$
or to the left of the smallest element or to the right of the largest element.
So to recover the pairs in $D$ we may assume that $I_+=I_-=\varnothing$ and the reduced form of $s$ is empty (in particular,
$\wt(s)=0$, where $\wt(s)$, by definition, is the difference between the number of $-$'s and the number of $+$'s).
Set $I^+:=\{i: t_i=+, s_i=-\}, I^-:=\{i: t_i=-,s_i=+\}$. Of course, the cardinalities of $I^+$ and $I^-$
have to be the same. Also,  any pair $p^i$ in $D$ either has the left element in $I^+$  and the right element in
$I^-$ (this is precisely the case when a pair is switchable)
or has both elements outside $I^+\sqcup I^-$. Let us explain how to recover the pairs lying in $I^+\sqcup I^-$.
We pair the elements in $I^+$ with those in $I^-$ using the following recipe repeatedly: if we have elements $i\in I^+, i'\in I^-$
with $i<i'$ such that all elements of $I^+\sqcup I^-$ between $i,i'$ has already been paired, then we pair $i$ with $i'$. This is again the standard recipe of canceling brackets, now with a $+$ being a ``(", and a $-$ being a ``)".
Now let us explain how to pair the elements in $\{1,\ldots,n\}\setminus I^+\setminus I^-$. For an already constructed
pair $p=\{i,i'\}$ with $i\in I^+, i'\in I^-$ define a subset $I^p\subset\{i+1,\ldots,i'-1\}$
of all indexes that do not lie between the elements of any other pair $p'\subset (I^+\sqcup I^-)\cap \{i+1,\ldots,i'-1\}$.
Then, clearly, $I^+\sqcup I^-\sqcup\bigsqcup_{p}I^p=\{1,\ldots,n\}$. All remaining  pairs are contained in exactly one
$I^p$ so it is enough to explain the pairing in the case when $I^p=\{1,\ldots,n\}$. This is again the bracket cancelation
rule, this time with a $-$ being a ``(" and a $+$ being a ``)",  we just pair the two brackets corresponding to each other.

We see that this algorithm produces a unique division $D$ with $s=t^D$ (and if the algorithm fails -- the brackets
cannot be canceled -- then $s$ does not have the form $t^D$).

%\subsection{Technical results}
\subsection{Consequences of splitting}\label{SS_split_cons}
We are going to prove our claims by induction on the size  $n$ of a basic categorification.
In the proof  we will extensively use the
splitting results of  Section \ref{S_cat_spl}. Recall that we have a subcategory $\Cat_-\subset \Cat$
spanned by the simples of the form $L(t-)$ with the inclusion functor $\iota:\Cat_-\hookrightarrow \Cat$
and the quotient category $\Cat_+$ with the quotient functor $\pi:\Cat\twoheadrightarrow \Cat_+$.
Then the left adjoint  $\pi^!$ induces an equivalence of $\Cat_+^\Delta$ with the subcategory
of $\Cat^\Delta$ of all objects, whose successive filtration quotients have the form $\Delta(t+)$.
Also $\iota^!$ is an exact functor $\Cat^\Delta\twoheadrightarrow \Cat_-^\Delta$. We have an exact sequence
of functors $0\rightarrow \pi^!\pi\rightarrow \operatorname{id}\rightarrow \iota\iota^!\rightarrow 0$
on $\Cat^\Delta$, see Subsection \ref{SS_hw_reminder}.

\begin{Lem}\label{Lem:ind_t+}
If the claim of Theorem \ref{Thm_proj_resol} holds for $t\in \{+,-\}^{n-1}$, then it also holds for $t+$.
\end{Lem}
\begin{proof}
From the  paragraph preceding the lemma,
it follows that we can treat a minimal projective resolution of
$\Delta(t)$ as that of $\Delta(t+)$.
So we only need to present a bijection between the divisions of $t$ and of $t+$ that preserves
the function $s$. Let $D=(I_+,I_-, p^1,\ldots,p^k)$ be a division of $t$. Define a division $\tilde{D}$
of $t+$ as follows. If $I_-=\varnothing$, set $\tilde{D}=(I_+\sqcup\{n+1\}, \varnothing, p^1,\ldots,p^k)$.
So suppose that $I_-\neq \varnothing$. Then let $j$ be the largest element of $I_-$. We set
$\tilde{D}=\{I_+, I_-\setminus \{j\}, p^1,\ldots,p^k, \{j,n+1\}\}$. It is easy to see that
$\tilde{D}$ is a division of $t+$. Moreover, it is not difficult to check that the map $D\mapsto \tilde{D}$
is a bijection between the sets of divisions and that this bijection preserves the function $s$.
\end{proof}

So it only remains to check the claim of Theorem \ref{Thm_proj_resol} for elements of the form $t-$.
We still have the full control over ``$\ldots-$-part''.

\begin{Lem}\label{Lem:ind_t-}
Suppose the claim of Theorem \ref{Thm_proj_resol} holds for $t\in \{+,-\}^{n-1}$. Let $P_\bullet$ be a minimal
resolution for $\Delta(t-)$. Then $\iota^!(P_\bullet)$ is a minimal resolution for $\Delta(t)(=\Delta_-(t))$.
In particular, the part of the minimal resolution of $\Delta(t-)$ consisting of the projectives of the form
$P(s-)$ is as specified in Theorem \ref{Thm_proj_resol}.
%In particular, $\Ext^i(\Delta(t-), L(s-))=\Ext^i(\Delta(t),L(s))$ for all $s$.
\end{Lem}
\begin{proof}
The claim that $\iota^!(P_\bullet)$ is a minimal projective resolution is equivalent to
$\Ext^i(\Delta(t-), L(s-))=\Ext^i(\Delta(t),L(s))$ for all $s$ and $i$. The latter follows
from Lemma \ref{Lem:hw_subcat}. To show the remaining statement it remains to produce a bijection
between the divisions of $t-$ with  $n\in I_-$ and the divisions of $t$.
Bijections just include/delete $n$ from $I_-$.
%We will prove a more general claim: if $M\in \Cat^\Delta$ has a minimal projective resolution $P_\bullet$,
%then $\iota^!(P_\bullet)$ is a minimal projective resolution for $\iota^!(M)$. First of all, $\iota^!(P_\bullet)$
%is a complex of projectives, it resolves $\iota^!(M)$ because of the exactness of $\iota^!$ on $\Cat^\Delta$ (and an easy
%fact that the kernel of a surjection of standardly filtered objects is also standardly filtered).
%To show the minimality of $\iota^!(P_\bullet)$ it is enough to check that $\head(\iota^!(P_0))=\head(\iota^!(M))$,
%where $\head$ means the maximal semisimple quotient.
%This follows from $\head(P_0)=\head(M)$ and the easy observation that $\iota^!$ commutes with $\head$ on $\Cat^\Delta$.
\end{proof}

Recall the equivalence  $\Eun:\Cat_+\rightarrow \Cat_-$. We also will write $\Eun$ for $\iota\circ \Eun\circ \pi:\Cat_+^\Delta
\rightarrow \Cat_-^\Delta$ so that $\Eun(\Delta(t+))=\Delta(t-)$. Next, we write $E_+=\pi^!\pi E$. As we have seen in the 
previous section, under the embedding $\Cat^\Delta_+\subset \Cat_+$ the functor $E_+$ corresponds to the categorification  
functor $\underline{E}$ on $\Cat_+$. So, for any $M\in \Cat_+^\Delta$, we have the exact sequence
\begin{equation}\label{eq:main_exact_proj}
0\rightarrow E_+M\rightarrow EM\rightarrow \Eun M\rightarrow 0.
\end{equation}
In particular, we have the exact sequence $0\rightarrow E_+ \Delta(t+)\rightarrow E\Delta(t+)\rightarrow \Delta(t-)\rightarrow 0$.
Let $P^+_\bullet$ denote the minimal projective resolution for $\Delta(t+)$. Then we can consider the double complex $E_+ P^+_\bullet\rightarrow E P^+_\bullet$, whose cone is a (usually non-minimal) projective resolution for $\Delta(t-)$.

Recall the involution $t\mapsto \bar{t}$ on $\Lambda=\{+,-\}^n$ and the naive dual highest weight $\sl_2$-\!\! categorification $\bar{\Cat}$.
We can view  $\bar{\Cat}_+$ as a subcategory in $\Cat$, it is a Serre subcategory spanned by the simples of the form
$L(+t)$. Similarly, we can view $\bar{\Cat}_-$ as a quotient of $\Cat$, the category $\bar{\Cat}_-^\Delta$ is identified
with the full subcategory in $\Cat^\Delta$ of all objects with a filtration whose successive quotients are of the form $\Delta(-t)$.

This discussion and Lemma \ref{Lem:ind_t-} have the following corollary.

\begin{Lem}\label{Lem:proj_bar}
Let $\Cat$ be a basic categorification of size $n$.
\begin{itemize}
\item Suppose that Theorem \ref{Thm_proj_resol} holds for $t\in \{+,-\}^n$ and the category $\bar{\Cat}$. Then it also
holds for $\bar{t}$ and the category $\Cat$.
\item Now suppose $t\in \{+,-\}^{n-1}$ and Theorem \ref{Thm_proj_resol} holds for any basic categorification of size $n-1$.
Then it also holds for $-t, t+\in \{+,-\}^n$ and the categorification $\Cat$.
\item
Further, if $P_\bullet$ is a minimal resolution for $\Delta(+t)$ in $\Cat$, then
$\bar{\iota}^!(P_\bullet)$ is a minimal resolution for $\Delta(t)$ in $\underline{\bar{\Cat}}$.
In particular, $\Ext^i(\Delta(+t),L(+s))=\Ext^i(\Delta(t),L(s))$ for all $s$.
\end{itemize}
\end{Lem}

\subsection{Character formulas}
Starting from this subsection we assume that Theorem \ref{Thm_proj_resol} is proved for all basic categorifications
of size $n'$ with $n'<n$. Our goal here is two-fold. First, we prove a result establishing  character formulas
for any categorification of size $n$. Second, we decompose the projectives $E P(t)$ into the sum of indecomposables.

For $j\in \{1,2,\ldots,n\}$ and $s\in \{+,-\}^n$,
set $h^j_\pm(s):= |I_\pm(s_j,\ldots,s_n)|$, where, recall, $I_+(s_j,\ldots,s_m),I_-(s_j,\ldots,s_m)\subset \{j,\ldots,m\}$
are the reduced signature of $(s_j,\ldots,s_n)$. %We write $h_\pm(s):=h^1_\pm(s)$.

\begin{Prop}\label{Prop:character}
Let $\Cat$ be a basic categorification of size $n$ and $t\in \{+,-\}^n$.
\begin{itemize}
\item[(i)]
On the level of the Grothendieck groups we have the equality predicted by Theorem \ref{Thm_proj_resol}. That is,
we have $[\Delta(t)]=\sum_{D} (-1)^{s(D)}[P(t^D)]$, where the sum is taken over all divisions $D$ of $t$.
\item[(ii)] Let $I$ be the set of all indexes $i$ such that $t_i=+$ and $h_+^{i+1}(t)=0$. For $i\in I$, let $t^i$ denote the element of $\Lambda= \{+,-\}^n$ obtained from $t$ by replacing the $+$ in the $i$th position by a $-$.  Then we have
\begin{equation}\label{eq:E_proj_decomp}
EP(t)=\bigoplus_{i\in I} P(t^i)^{\oplus (h^i_-(t)+1)}.
\end{equation}
\end{itemize}
\end{Prop}

\begin{Rem}\label{Rem:part2_equiv}
The statement of part (ii) has several equivalent reformulations. First, applied to $\bar{\Cat}$, it gives
\begin{equation}\label{eq:F_proj_decomp}
FP(t)=\bigoplus_{i\in \bar{I}}P(\bar{t}^i)^{h_+(t_1,\ldots,t_i)+1},\end{equation} where $\bar{I}$ consists of all indices
$i$ such that $t_i=-, h_-(t_1,\ldots,t_{i-1})=0$, and $\bar{t}^i$ is obtained from $t$ by switching the $i$th
component from $-$ to $+$. Also, thanks to the biadjointness of $E,F$, we see
that (\ref{eq:E_proj_decomp}) is equivalent to the claim that the simple subquotients of
$FL(t)$ are $L(\bar{t}^i)$, where $i$ is such that $t_i=-, h_+^{i+1}(t)=0$, with multiplicity $h_-^{i+1}(t)+1$.
Similarly,  the irreducible subquotients of $E L(t)$ are precisely $L(t^i)$ with $t_i=+, h_-(t_1,\ldots,t_{i-1})=0$, with multiplicity $h_+(t_1,\ldots,t_{i-1})+1$.
\end{Rem}

\begin{proof} We proceed in several steps.

{\it Step 1.} We claim that proving (i) amounts to proving (ii) and vice versa. More precisely, (i) holds for all elements
$t$ with $\wt(t)=w$ if and only if (\ref{eq:E_proj_decomp}) holds for all $t$ with $\wt(t)=w-2$. Let us prove this.

(i) is equivalent to saying that the classes $[P(t)]$ constitute Lusztig's canonical basis of
the tensor product $(\Q^2)^{\otimes n}$ (with factors ordered right to left).
To see this we notice that the coefficients of in the expression of $\Delta$'s via $P$'s in (i)
are values at $1$ of the corresponding parabolic Kazhdan-Lusztig polynomials, compare with \cite{BS1}.

Using the observation in the previous paragraph, we show that (i) implies (ii).
The classes $[P(t)]$ form a canonical basis if and only if  the classes $[L(t)]$ form the dual canonical basis
(with tensor factors ordered left to right).
Let $L_t\in (\Q^2)^{\otimes n}$ denote the element of  the dual canonical basis corresponding to $t$
(as in \cite{FKK} but we use the opposite sign convention). We will use the inductive description of the dual
canonical basis elements given in \cite[Theorem 3.1]{FKK} to show that $f L_t$ is decomposed
in the dual canonical basis as described in Remark \ref{Rem:part2_equiv}, this will show that (i) implies (ii). We will argue
by induction on the size $n$.  If $t=t(1)-+t(2)$, where the length of $t(1)$ is $k$, then, according to \cite[Theorem 3.1]{FKK},
$L_t=\iota_k(L_{t(1)t(2)})$, where $\iota_k:(\Q^2)^{\otimes n-2}\rightarrow (\Q^2)^{\otimes n}$
is the map defined by $\iota(u_1\otimes u_2)=u_1\otimes (v_-\otimes v_+-v_+\otimes v_-)\otimes u_2$ with $u_1\in (\Q^2)^{\otimes k}$.
So $f L_t=\iota_k (f L_{t(1)t(2)})$. It follows that the multiplicity of $L_{s(1)s(2)}$ in $f L_{t(1)t(2)}$
equals that of $L_{s(1)-+s(2)}$ in $fL_{t(1)-+t(2)}$. The set of indexes $i$ such that $t_i=-$ and $h_+^{i+1}(t)=0$
for $t$ is in a natural bijection with such a set for $t(1)t(2)$ (a part of the indexes stays the same,
while the other is decreased by $2$) and the bijection preserves the functions $h_-^{i+1}$.
So if the decomposition specified in Remark \ref{Rem:part2_equiv} holds for $fL_{t(1)t(2)}$,
then it also holds for $fL_{t(1)-+t(2)}$.

%From here it is easy to see that if the required formula holds for $L_{t_1t_2}$,
%then it also holds for $L_{t_1-+t_2}$.

This reduces  the computation of  the decomposition of $f L_t$
to the case when $t=+\ldots+-\ldots-$ (say, $k$ pluses). In this case \cite[Theorem 3.1]{FKK} implies that
$L_t=v_+^{\otimes k}\otimes v_-^{\otimes n-k}$. We have $\bar{t}^i=+\ldots+-\ldots-+-\ldots-$ (with a $+$
on the position $k+i$). The multiplicity of $L_{\bar{t}^i}$ in $fL_{t}$ predicted by Remark \ref{Rem:part2_equiv}
equals $n+1-k-i$. We have $L_{\bar{t}^1}=v_+^{\otimes k+1}\otimes v_-^{\otimes n-k-1}$ and we can compute
$L_{\bar{t}^i}$ for $i>1$ as described in the previous paragraph: we get
$$L_{\bar{t}^i}=v_+^{\otimes k}\otimes v_-^{\otimes i-1}\otimes v_+\otimes v_-^{n-k-i}-v_+^{\otimes k}\otimes v_-^{\otimes i-2}\otimes v_+\otimes v_-^{n-k-i+1}.$$
An easy computation shows that indeed $f L_t=\sum_{i}(n+1-k-i)L_{\bar{t}^i}$. We have checked that
(i) implies (ii).
%We can get formulas for $L_{\bar{t}^i}$
%from \cite{FKK} and use them to check the required equality for $f L_t$.
%This yields (ii).

Let us now show that (ii) implies (i). By our inductive assumption, (i) and the analog of (ii) hold in
$\Cat_+$ (and also in $\Cat_-$). So (ii) (applied to
$s+\in \underline{\Lambda}_+$) uniquely determines the character of $P(s-)$. As the previous two paragraphs show,
(i) for $\Cat_-$ together with (ii) implies that the classes of $L(t)$ in $(\Q^2)^{\otimes n}$ form the dual
canonical basis. Therefore the  elements $[P(t)]\in (\Q^2)^{\otimes n}$ are forced to be the elements of
Lusztig's canonical basis. As we have remarked above, this is equivalent to (i).
%The point is that both (i) and (ii) hold for the standard given
%categorifcation defined in the end of Subsection \ref{SS_hw_ex}.
%(i) follows, for example, from \cite{BS}, while (ii) stems from the fact
%that the images of the projectives in the Grothendieck group $[\Cat]=(\Q^2)^{\otimes n}$ is Lusztig's canonical
%basis, see, for example, \cite{BFK}. Now both (i) and (ii) determine the decomposition of the images of projectives via the images of the %standards
%uniquely. Presumably, the equivalence of (i) and (ii) can be established in purely combinatorial terms but we do not need
%that.

{\it Step 2.} Also let us point out that (i) for $t$ and $\Cat$ is equivalent to (i) for $\bar{t}$ and $\bar{\Cat}$.
Since $\wt(\bar{t})=-\wt(t)$, it is enough to establish (i) for all $t$ with $\wt(t)\leqslant 0$, i.e., when the number
of $+$'s is bigger than or equal to the number of $-$'s.

{\it Step 3.} Thanks to the results of Subsection \ref{SS_split_cons}, (i) holds for all elements $t$ of the form $-s$ and $s+$
with $s\in \{+,-\}^{n-1}$. So it remains to prove (i) for all $t$ of the form $+s-, s\in \{+,-\}^{n-2}, \wt(s)\leqslant 0$.

{\it Step 4.} We will prove (i) for the elements of the form $+s-$ and (ii) for the elements
of the form $+s+$ with $\wt(s)=w$ by using the increasing induction on $w$.
We will use the observation that if $P_\bullet$ is the projective resolution
for $\Delta(+s+)$, then the cone of $E_+P_\bullet\rightarrow E P_\bullet$
is a projective resolution of $\Delta(+s-)$.

The induction base is (ii) for $w=2-n$.  Here we have $P(+\ldots+)=\Delta(+\ldots+)=L(+\ldots+)$,
$P(+\ldots+-)=E\Delta(+\ldots+)$ (indeed, the right hand side has simple head equal to $L(+\ldots+-)$, 
this follows, for example, from the main result of \cite{cryst} and the equality $\Delta(+\ldots+)=L(+\ldots+)$).
Steps 1,3 together imply that (i) for $\wt(s)=w$ is equivalent to (ii) for $\wt(s)=w-2$.

Let $P(u-)$ be a direct summand of $EP(+s+)$. Then $u=+s$ and this summand occurs with multiplicity 1.
This follows from the claim that the functor $\mathcal{E}: \Cat_+\rightarrow \Cat_-$, see Subsection \ref{SS_Eun},
is a category equivalence.

Assume, for a moment, that $P(-s'+)$ is not a direct summand of $EP(+s+)$ for any $s',s$.
Then all direct summands of $EP(+s+)$ are of the form $P(+u)$. The multiplicity of $P(+u)$
in $EP(+s+)$ equals $\dim \Hom(E P(+s+), L(+u))=\dim \Hom(P(+s+),FL(+u))$ and the latter equals
to the multiplicity of $L(+s+)$ in $F L(+u)$. But both objects lie in $\bar{\Cat}_+$ that is closed
with respect to $F$. So we conclude that the multiplicity of $P(+u)$ in $EP(+s+)$ equals to the multiplicity
of $P(u)$ in $EP(s+)$. The latter is given by (ii) thanks to the inductive assumption. So (ii) for $+s+$ follows.
The class of $E_+ P_\bullet$ is as predicted by (ii) by our assumptions on $n$. The observation on
a projective resolution of $\Delta(+s-)$ above implies (i) for $+s-$. So it remains to show that
$P(-s'+)$ is not a direct summand of $EP(+s+)$.

%To prove that the class of $E P_\bullet$
%is as predicted by (ii) it is  enough, by the previous subsection, to check that a projective of the form
%$P(-s'+)$ is not a direct summand of $E P(+s+)$. Indeed, then we know all multiplicities of the indecomposable projectives
%in $E P(+s+)$ thanks to Lemmas \ref{Lem:ind_t-}, \ref{Lem:proj_bar}, and they are consistent with (ii).

{\it Step 5.} Let us show that every indecomposable projective of the form $P(t+)$ that appears in
$E P(+s+)$ also appears in $E_+ P(+s+)$. Choose any filtration on $EP(+s+)$ whose successive quotients
are standards ordered in the increasing order from top to bottom. Then $E_+ P(+s+)=\pi^!\pi (E P(+s+))$ is just the maximal filtration
 component of $EP(+s+)$ lying in $\Cat^\Delta_+$ and this
description of $E_+ P(+s+)$ is independent of the choice of a filtration on $EP(+s+)$, see Lemma \ref{Lem:hw_quot}.
This independence implies the
claim, since $P(t+)$ is in $\Cat_+^\Delta$.

{\it Step 6.} Step 5 implies that the only  simple of the form $L(-s'+)$ in the head of $E P(+s+)$ also lies in
 the head of $E_+ P(+s+)$. But the inductive assumption describes all simples in the latter and the only simple
 of the form $L(-s'+)$ that may appear there is $L(-s+)$.
 So we only need to show that $L(-s+)$ never appears in the head of $E P(+s+)$,
as long as $\wt(s)\leqslant 0$. The top of $P(+s+)$ is the simple $L(+s+)$, and we have $h_-(L(+s+))=\max(h_-(s)-1,0)$.
Recall that for an element $c$ of an $\sl_2$-crystal $C$ with crystal operators $\tilde{e},\tilde{f}$  we write $h_-(c)$
for the maximal number $N$ such that $\tilde{f}^N c\neq 0$; the number $h_+(c)$ is defined similarly but for $\tilde{e}$
instead of $\tilde{f}$. In the crystal $\{+,-\}^n$ (see, e.g., \cite[Example 2.4]{cryst}) we have $h_\pm(c)=h^1_\pm(c)$.
We remark that $\wt(c)=h_-(c)-h_+(c)$.

By \cite[Lemma 5.11]{CR}, for all simples $L$ in the head of $E P(+s+)$, we have $h_-(L)=h_-(L(+s+))+1$.  \cite[Theorem 5.1]{cryst} implies that $h_-(L(+s+))+1=h_-(+s+)+1=\max(h_-(s),1)$.
So it remains to show that $h_-(L(-s+))<\max(h_-(s),1)$ and this is where we are going to use the assumption $w=\wt(s)\leqslant 0$.
We have $h_-(s)-h_+(s)=w$, so either $h_+(s)>0$ or $h_+(s)=h_-(s)=0$. If $h_+(s)>0$, then the first $-$ in $-s+$ does not
survive in the reduced form, hence $h_-(-s+)=h_-(+s+)=\max(h_-(s)-1,0)<\max(h_-(s),1)$. If $h_+(s)=h_-(s)=0$, then
$h_-(-s+)=0<1=\max(h_-(s),1)$. We get a contradiction either way. This completes the proof of the proposition.
\end{proof}

\begin{Rem}
Part (i) of the previous proposition gives the multiplicities of simples in standards in a basic categorification.
Thanks to the existence of a family filtration, see Subsection \ref{SS_family_filtr}, this yields formulas for $[\Delta(\lambda):L(\mu)]$ when $\lambda$ and $\mu$ are in the same family
for an arbitrary highest weight $\sl_2$-categorification. These formulas generalize Theorem 1.10 in \cite{Kleshchev1}.
Similarly, the formulas for $[FL(t)],[EL(t)]$ in Remark \ref{Rem:part2_equiv} generalize (iv) of Theorems B,B'
in \cite{BK_functors}.
\end{Rem}

\subsection{Equivalent formulation}
Thanks to Lemma \ref{Lem:ind_t-} we know the part of a minimal projective resolution of $\Delta(t-)$ consisting of
the projectives of the form $P(s-)$.
Now we want to describe the occurrences of a given projective $P(s+),s\in \{+,-\}^{n-1}$ in the double complex
$E_+ P_\bullet\rightarrow E P_\bullet$, where $P_\bullet$ is a minimal projective resolution of $\Delta(t+)$.
We will see that, whenever $P(s+)$ occurs (with a nonzero multiplicity) in the double complex, exactly one of
the following two options takes place.
\begin{itemize}
\item[(A)] $s+$ has the form $(t-)^D$ for some division $D$ of $t-$. Then $P(s+)$ occurs only in the homological degree
$s(D)-1$ in  $E_+ P_\bullet$ with some multiplicity $m$ and (if $m>1$) also occurs  in the homological degree
$s(D)-1$ in $E P(t)$ with multiplicity $m-1$. Moreover, the corresponding map $\Hom(E P_{s(D)-1}, L(s+))\rightarrow
\Hom(E_+ P_{s(D)-1}, L(s+))$ is injective.
\item[(B)] $s+$ does not have the form in (A) (but $P(s+)$ still occurs in $E_+ P_\bullet$). Then there is $d$ such that
$P(s+)$ occurs only in  $E_+ P_{d}, E_+ P_{d+1}$ with multiplicities $m,m-1$, where $m>1$.
Also $P(s+)$ occurs with multiplicities $m-1,m-2$ in $E P_d, E P_{d+1}$, respectively. The maps $\Hom(E P_{i}, L(s+))\rightarrow
\Hom(E_+ P_{i}, L(s+))$ are injective for both $i=d,d+1$.
\end{itemize}

Modulo the previous claim, Theorem \ref{Thm_proj_resol}, for a given element $t-\in \{+,-\}^n$ and any basic categorification $\Cat$,
is equivalent to the following statement
\begin{itemize}
\item[(*)] For any $s$ as in (B) the image of the  map $\Hom(E_+ P_{d}, L(s+))\rightarrow \Hom(E_+ P_{d+1}, L(s+))$
is not contained in the image of the map $\Hom(E P_{d+1}, L(s+))\rightarrow \Hom(E_+ P_{d+1}, L(s+))$.
\end{itemize}

Indeed, from (A) it follows that if $s+=(t-)^{D}$, then $\Ext^i(\Delta(t-), L(s+))=\C$ if $i=s(D)$ and 0 else.
On the other hand, modulo (B) the claim (*) is equivalent to $\Ext^i(\Delta(t-), L(s+))=0$ for all other $s$'s.

To establish that exactly one of (A) and (B) holds we need to show that the following holds:
\begin{itemize}
\item If $s$ is as in (A), then there is only one indecomposable projective $P(s'+)$ in the complex $P_\bullet$
such that $P(s+)$ is a summand of $E_+ P(s'+)$. The projective $P(s'+)$ occurs in homological degree $s(D)-1$.
\item Suppose $s$ is not as in (A) but  takes the form $(t^D)^i$ for some division $D$ of $t$ with $t^D_i=+$ and
$h_+^{i+1}(t^D)=0$. Then there are two possible divisions $D=D_+,D_-$ with this property. They satisfy $s(D_-)=s(D_+)+1$.
The multiplicity of $P(s+)$ in $E_+ P\left(t^{D_+}+\right)$ is bigger by $1$ than the multiplicity in
$E_+ P\left( t^{D_-}+\right)$.
\end{itemize}
The claim comparing the multiplicities in $E_+ P_\bullet$ and $E P_\bullet$ is (ii) of Proposition \ref{Prop:character},
while the injectivity of the corresponding maps in (A),(B) was basically established in Step 5 of the proof of that
proposition. Also it is clear that any label $s$ such that $P(s+)$ appears in $E_+ P_\bullet$ has the form
$(t^D)^i$ with $i$ subject to $t^D_i=+, h^{i+1}_+(t^D)=0$.

%Below we will need a combinatorial lemma that characterizes elements of the form $t^D$.
%
%\begin{Lem}\label{Lem:comb_divis} Let $t,\hat{t}\in \{+,-\}^{n-1}$ and let $(\hat{I}_+,\hat{I}_-)$ be the reduced
%form of $\hat{t}$. The element $\hat{t}$ has the form $t^D$ for some division $D$ of $t$ if and only if for any two consecutive
%elements $i_1<i_2$ of $\hat{I}_+\sqcup \hat{I}_-$ the following hold.
%\begin{itemize}
%\item Let $I^+_{i_1,i_2}=\{i\in \{i_1+1,\ldots, i_2-1\}| t_i=+, \hat{t}_i=-\}, I^-_{i_1,i_2}=\{i\in \{i_1+1,\ldots, i_2-1\}| t_i=-, %\hat{t}_i=+\}$
%\item Let $(I^{i_1,i_2}_+, I^{i_1,i_2}_-)$ be the reduced form of $(t_{i_1+1}, \ldots t_{i_2-1})$. Then $|I^{i_1,i_2}_+|=|I^{i_1,i_2}_-|$, %$\hat{t}_i=-$ for $i\in I^{i_1,i_2}_+$, $\hat{t}_i=+$ for $i\in I^{i_1,i_2}_-$.
%\end{itemize}
%\end{Lem}

We proceed to describing divisions $D$ giving the same $(t^D)^i$ (with different $i$). 

Let a division $D$ of $t$ be of the form $(I_+,I_-,p^1,\ldots, p^k)$. We say that a pair $p^i=(j,j'), j<j'$ is {\it external} if
there is no pair $p^{i_1}=(j_1,j_1')$ with $j_1<j<j'< j_1'$. The condition $h^{i+1}_+(t^D)=0$ holds if and only if
$i$ is the largest element of $I_+$ or the larger element in an external pair $p$ that lies to
the right  of $I_+$. In the first case, the reduced form $I_\pm\left((t^D)^i\right)$
satisfies $I_+\left((t^D)^i\right)=I_+\setminus\{i\}, I_-\left((t^D)^i\right)=I_-\sqcup\{i\}$.
In the second case,  $I_+\left((t^D)^i\right)=I_+$ and $I_-\left((t^D)^i\right)=I_-\sqcup p$.

By the previous paragraph, $I_+(s)\subset I_+$ and
$I_-\subset I_-(s)$. Furthermore,
\begin{enumerate}\item either $I_+$ is obtained from
$I_+(s)$ by adding the smallest element of $I_-(s)$ \item or coincides with $I_+(s)$.
\end{enumerate}
In the first case the set $I_-$ is obtained from $I_-(s)$ by deleting the smallest element.
In the second case, $I_-$ is obtained from $I_-(s)$ by deleting two consecutive elements that
form a pair in $D$.

Now let $D_l=(I_{+,l},I_{-,l},p_l^1,\ldots,p_l^{k_l}),l=1,2,$ be two different divisions such that $s$ is obtained from both $t^{D_1},t^{D_2}$ by replacing a suitable $+$ with a $-$.  We cannot have $I_{-,1}=I_{-,2}$. Indeed, the previous analysis implies that in this case $t^{D_1}=t^{D_2}$ which, in turn, implies $D_1=D_2$.

Further,  we claim that \begin{equation}\label{eq:ineq_I}I_{-,1}\cup I_{-,2}\neq I_-(s).\end{equation} Assume the converse.
We have $t_i=-$ for any $i\in I_-(s)$ because $t_i=-$ for any $i\in I_{-,l}, l=1,2$. Also we see that for any two consecutive
elements $j,j'$ of $I_+(s)\sqcup I_-(s)$ we have $\wt(t_{j+1},\ldots, t_{j'-1})=0$ because there is $l=1,2$
such that $j,j'$ are consecutive elements of $ I_{+,l}\sqcup I_{-,l}$.
Finally, we can assume that $I_{+,1}=I_+(s)$. So we see that $\wt(t)=|I_{-,1}|-|I_{+,1}|$,
while the preceding discussion shows that $\wt(t)=|I_-(s)|-|I_+(s)|$, a contradiction.

We deduce from (\ref{eq:ineq_I}) that there may be no more than two divisions $D_1,D_2$ such that $s$ is obtained from $t^{D_1},t^{D_2}$.
We have only one $D$ if and only if $t_i=+$ for the largest element $i$ in $I_-(s)$. This is precisely the
case when $s+=(t-)^{D'}$ for some division $D'$ of $t-$. More precisely, if $D=(I_+,I_-, p^1,\ldots,p^k)$,
then $D'$ equals $(I_+\setminus \{i\},I_-\setminus \{i\},p^1,\ldots,p^k, \{i,n\})$.

Now consider the  case when $t_i=+$ for some non-maximal
element $i$ of $I_-(s)$. We have two divisions $D_+,D_-$ such that $s$ is obtained from $t^{D_+},t^{D_-}$
and $s(D_+)=s(D_-)-1$. We remark that $t^{D_+}$ and $t^{D_-}$ are different just in 2 positions that are
elements of $I_-(s)$ (and one of these positions is $i$). Moreover, $t^{D_-}=(t^{D_+})^{D_0}$, where $D_0$ is a division of $t^{D_+}$
with $s(D_0)=1$: the division $D_0$ coincides with $D_-$ viewed as a division of $t^{D_+}$.

This completes the proof of the claim that exactly one of (A) and (B) holds.

%\subsection{Distinguished standard subquotient}
%Let $t,s\in \{+,-\}^{n-1}$ be such that $P(s+)$ appears in $E_+ P(t+)$.
%The goal of this subsection is to produce a filtration having  a standard subquotient $\Delta(s'+)$ of $P(t+)$ such that
%\begin{itemize}
%\item[(i)] $L(s+)$ is in the head of $E_+\Delta(s'+)$ but not in the head of $E\Delta(s+)$.
%\item[(ii)] The quotient $L(s+)$ of $E_+\Delta(s'+)$ gets to the head of $E_+ P(t+)$.
%\end{itemize}
%Namely, recall that $(I_+(s),I_-(s))$ is the reduced form of $s$. Let $i$ be the largest element
%of $I_-(s)$. Let $s'$ be the element obtained from $s$ by replacing the $-$ on the $i$th position
%with a $+$.
%
%We will now explain what filtration and subquotient we need. Let $j$ denote the only element of $\{1,\ldots,n-1\}$
%with  $t_{j}=+, s'_{j}=-$.
%
%Consider first the easy situation, when $i=n-1$ and $j=1$ so that $t=+t'-, s'=-t'+$. Consider

\subsection{Proof of the main theorem}
Recall that it is enough to prove Theorem \ref{Thm_proj_resol} or, equivalently, the claim (*) for
an element  of the form $+t-\in \{+,-\}^n$, where $t\in \{+,-\}^{n-2}$ has weight $\wt(t)\leqslant 0$.

Let $s\in \{+,-\}^{n-1}$ be as in (B). We claim that $1\not\in I_-(s)$. Assume the contrary.
Let $j$ be the smallest element of $I_-(s)$ larger than $1$.
It follows that $h_+(s_2,\ldots,s_{j-1})=h_-(s_2,\ldots,s_{j-1})=0$. Also $h_+(s_{j+1},\ldots,s_{n-1})=0$
because $j\in I_-(s)$.
So we see that $\wt(s)\geqslant 2$ (the minuses on the positions $1,j$ make the impact of $2$).
It follows that $\wt(s+)>0$. But $\wt(s+)=\wt(+t-)=\wt(t)\leqslant 0$, a contradiction.

So let $i>1$ be the minimal element of $I_-(s)$. We remark that $s_l=t^{D_\pm}_l$ for $l<i$, where the divisions $D_{\pm}$
of $t$ were introduced in the previous subsection. For an element $s'\in \{+,-\}^{n-1}$ with
$s'_1=s_1, \ldots, s'_{i-1}=s_{i-1}$ we set $\underline{s}':=(s'_i,\ldots, s'_{n-1})$. Also for a division $D$ of $s'$ we write
$\underline{D}$ for the induced division of $\underline{s}'$ provided the latter makes sense. Let us write $\underline{E},\underline{E}_+$
for the categorification functors for  basic categories of  sizes $n-i, n-i-1$.  To prove (*) it is enough to show that
there are identifications $\Hom(E_+ P(t^{D_\pm}+), L(s+))\xrightarrow{\sim} \Hom(\underline{E}_+ P(\underline{t}^{\underline{D}_\pm}+), L(\underline{s}+)),\Hom(E P(t^{D_\pm}+), L(s+))\xrightarrow{\sim} \Hom(\underline{E} P(\underline{t}^{\underline{D}_\pm}+), L(\underline{s}+))$  that intertwine the natural maps between the Hom spaces. Indeed, then (*) will follow from the
induction assumptions on $n$.

The claim boils down to showing the following: take a simple $L(s)$
and a standardly filtered object $M$. Assume that if $s_1=-$, then $M\in \bar{\Cat}_-$.
Let $\underline{E}$ denote the functor induced by $E$
on $\bar{\Cat}_{s_1}^{\Delta}$. Let $s'$ be defined by $s=s_1s'$. Then $\Hom(\underline{E}M, L(s'))$ is naturally identified with
$\Hom(E M, L(s))$. The claim follows from the observation that $\underline{E}M$ is a quotient of $EM$
and there are no nonzero homomorphisms from the kernel to $L(s)$. Applying an easy induction we
prove a generalization of the previous claim to any starting sequence in $s$. That claim and the naturality
of the identification now imply the statement in the previous paragraph.

The proof of Theorem \ref{Thm_proj_resol} is now complete.

\subsection{Applications to arbitrary categorifications}
First, we can get some information about  the modules $E P(\lambda), F P(\lambda)$ generalizing
Proposition \ref{Prop:character}, (ii).

\begin{Prop}\label{Prop:proj_arb}
Let $\Cat$ be a highest weight $\sl_2$-categorification with respect to a hierarchy structure on a poset $\Lambda$.
Pick a family $\Lambda_a$ and $\lambda\in \Lambda_a$. Set $t:=\sigma_a^{-1}(\lambda)$. Then $EP(\lambda)$ contains
$P(\sigma_a(t^i))$ as a direct summand with multiplicity  $h_-^i(t)+1$ for all indices $i$ such that
$t_i=+, h_+^{i+1}(t)=0$ (and does not contain $P(\sigma_a(s))$ for the other elements $s$).
Here, as before, $t^i$ stands for the element of $\{+,-\}^{n_a}$ such that
$t^i_j=t_j$ for $j\neq i$ and $t^i_i=-$. Similarly, $FP(\lambda)$ contains $P(\sigma_a(\bar{t}^i))$
with multiplicity $h_+(t_1,\ldots,t_i)+1$ for all indices $i$ such that $h_-(t_1,\ldots,t_{i-1})=0$
and $t_i=-$. Here $\bar{t}^i_j=t_j$ for $j\neq i$ and $\bar{t}^i_i=+$.
\end{Prop}
\begin{proof}
It is enough to prove the claim for $EP(\lambda)$, the claim for $F P(\lambda)$ is  obtained via passing to the dual categorification.
Thanks to the family filtration, see Subsection \ref{SS_family_filtr}, we can pass to a highest
weight quotient categorification $\Cat_1$ of $\Cat$ with the property that $\Lambda_a$
is a minimal family (recall a natural inclusion $\Cat_1^{\Delta}\subset \Cat^\Delta$).  Below
we assume that $\Cat=\Cat_1$.

Let $\iota$ denote the inclusion of the Serre subcategory $\Cat_a$ corresponding to
the poset $\Lambda_a$ into $\Lambda$ and $P_a(\lambda)$ be the indecomposable corresponding to $\lambda\in \Lambda_a$
in $\Cat_a$. Then $\iota^!(P(\lambda))=P_a(\lambda)$. Since $\iota^!$ commutes with $E$, we are done.
\end{proof}

The following proposition describes the head of $E\Delta(\lambda)$. This description generalizes
 Brundan's and Kleshchev's for the representations of $\GL$, see \cite{BK_functors}.

\begin{Prop}\label{Prop:E_stand}
Suppose $\lambda\in \Lambda_a, \lambda=\sigma_a(t)$. Then $\head(E\Delta(\lambda))=\bigoplus_{i}L(\sigma_a(t^i))$,
where $i$ is running over the set of all indices such that $t_i=+, h^{i+1}_+(t)=0$. Recall that $t^i$ is the $n_a$-tuple
obtained from $t$ by replacing the $i$th element with a $-$.
\end{Prop}
\begin{proof}
Because of  the standard filtration, only a simple of the form $L(\sigma_a(t^i))$ can appear in the head of $E\Delta(\lambda)$ and that the
multiplicity of every simple in the head is at most $1$. %Thanks to the existence of a family filtration,
%Subsection \ref{SS_family_filtr}, we may reduce the proof to the case when the categorification $\Cat$ is basic.
Thanks to Proposition \ref{Prop:proj_arb}, the only elements $\mu\in \Lambda_a$ such that
$P(\mu)$ appears  in $EP(\lambda)$ have the form $\mu=\sigma_a(t^i)$ with $h_+^{i+1}(t)=0$.
It follows that $\head(E \Delta(\lambda))\subset \bigoplus_{i}L(\sigma_a(t^i))$, where the summation is as in the statement of the
proposition. It remains to prove that every summand on the right hand side appears in  the left hand side. For this it is enough
to show that $\Ext^1(\Delta(\sigma_a(t^j)), L(\sigma_a(t^i)))=0$ for all $j>i$. Thanks to the existence of  a
family filtration, we can reduce the proof to the case when $\Lambda_a$
is a maximal family in $\Lambda$.   The minimal projective resolution of $\Delta(\sigma_a(t^j))$ is the same in
$\Cat$ and in the quotient corresponding to the family $\Lambda_a$. So we may assume that
$\Lambda=\Lambda_a$. Thanks to Theorem \ref{Thm_proj_resol}, we need to check that there is no division $D$ of $t^j$ with $s(D)=1$ and $t^i=(t^j)^D$. The $n$-tuples $t^j$ and $t^i$ differ only in 2 positions: $i$ and $j$ so $(i,j)$ is the only switchable pair in $D$.
However, this implies that $\wt(t_{i+1},\ldots,t_{j-1})=0$. Since $t_j=+$, this contradicts $h^{i+1}_+(t)=0$.
\end{proof}

\subsection{Structure of $EL(t)$}
In this subsection we will use the above results to get some information on the structure of $EL(t)$ for $t\in \{+,-\}^n$. The simple
subquotients of $EL(t)$ together with multiplicities were described in Remark \ref{Rem:part2_equiv}.

%»справить обозначени€ в лемме?
Let $s^1,\ldots,s^h$, where $h=h_+(t)$, be the elements of $\{+,-\}^n$ obtained from $t$ by replacing the
$i$th (from the left) $+$ in the reduced form of $t$ with a $-$. According to the previous lemma,
$L(s^1),\ldots, L(s^h)$ are precisely the irreducible constituents of $E L(t)$, $L(s^i)$ occurs
with multiplicity $i$. We remark that $h_-(s^h)=h_-(t)+1$ and $s^h=\tilde{e} t$, while
$h_-(s^i)=h_-(t)$ for $i<h$.

Now we are going to investigate a finer structure of $E L(t)$. We have an endomorphism $X$ of $E L(t)$ with $X^{h+1}=0$.
Set $N_i:=\ker X^i/\ker X^{i-1}$. Clearly, $X$ induces embeddings $N_{h}\hookrightarrow N_{h-1}\hookrightarrow\ldots\hookrightarrow N_1$.
Recall that by the radical filtration of an object $N$ one means the sequence $N=R_0(N)\supset R_1(N)\supset\ldots$ such that
$R_i(N)$ is the kernel of the map $R_{i-1}(N)\twoheadrightarrow \head(R_{i-1}(N))$. Dually, one introduces the coradical
filtration $\{0\}=R_0^*(N)\subset R_1^*(N)\subset\ldots$.

\begin{Prop}\label{Prop:EL}
The simple constituents of $N_i$ are $L(s^i),\ldots,L(s^h)$, each occurring with multiplicity $1$.
Furthermore, $R_j(N_i)$ coincides with $R^*_{h+1-i-j}(N_i)$ and has simple constituents $L(s^{i+j}),\ldots, L(s^h)$.
\end{Prop}
We remark that the last claim is equivalent to $R_{j-1}(N_i)/R_{j}(N_i)=L(s^{i+j})$ for all $j$
(or to $R_j^*(N_i)/R^*_{j-1}(N_i)=L(s^{h+1-j})$ for all $j$).
\begin{proof}
We prove this claim by induction on $n$. The case $n=1$ is trivial. Suppose that the claim is proved
for all basic categorifications of size $n-1$. Recall the subcategory $\Cat_-\subset \Cat$
and the quotient $\Cat_+$ of $\Cat$ with the projection $\pi:\Cat\rightarrow \Cat_+$.
Consider the simple $L_{t_n}(t)$ in $\Cat_{t_n}$ corresponding to
$L(t)$. Recall that on $\Cat_{t_n}$ we have, thanks to Section \ref{S_cat_spl}, categorification
functors $E_{t_n},F_{t_n}$, where $E_{t_n}$ is the functor induced
by $E$ (recall that $\Cat_-$ is $E$-stable). In particular, we have $\pi(E L(t))=E_{+}\pi(L(t))=E_+ L_+(t)$
and the morphism $X_+$ of $E_+ L_+(t)$ is induced from $X$.

We have $s^j_n=t_n$ for all $j$ if $t_n=-$ or if $t_n=+$ and $n\not\in I_+(t)$. If $t_n=-$,
then $EL(t)\cong E_{-}L_{-}(t)\in \Cat_-$ and we are done by induction, because the size
of $\Cat_{-}$ is $n-1$. If $t_n=+$ and $n\not\in I_+(t)$, then $\pi(L(s^j))=L_{+}(s^j)$
for all $j$. Moreover, $\pi(\ker X^i)\subset \ker X_+^i$. The operator on $\Hom(P(s^j), EL(t))$
induced by $X$ has the nilpotency degree not exceeding $j-1$. It follows that the multiplicity
of $L(s^j)$ in $\ker X^i$ is at least $j-i$. But the multiplicity of $L_+(s^j)$ in $\ker X_+^i$
is exactly $j-i$. So $\pi(\ker X^i)=\ker X_+^i$. Also $\pi(\ker X^i/\ker X^{i+1})=\ker X_+^i/\ker X_+^{i+1}$
and therefore $\pi(R_j(\ker X^i/\ker X^{i+1}))\supset R_j(\ker X_+^i/\ker X_+^{i+1})$. From
$$R_{j-1}(\ker X_+^i/\ker X_+^{i+1})/R_j(\ker X_+^i/\ker X_+^{i+1})=L_+(s^{i+j})$$ we deduce
$$R_{j-1}(\ker X^i/\ker X^{i+1})/R_j(\ker X^i/\ker X^{i+1})=L(s^{i+j}).$$

So it remains to consider the case when $t_n=+, n\in I_+(t)$. The difference here is that
$\pi(L(s^h))=0$. But still $\pi(L(s^j))=L_+(s^j)$ for $j<h$. First, let us remark that $R_1^*(\ker X)=L(s^h)$ (this is just the socle
of $EL(t)$). Second, the nilpotency degree of $X$ on $\Hom(P(s^h), E L(t))$ is $h-1$ (this follows from \cite[3.3.1, Proposition 5.20(c)]{CR}). Then the induction assumption applied to $E_+ L_+(t)$ leads to the proof of the proposition completely
analogously to the previous  paragraph.
\end{proof}

\section{Ringel duality and tiltings}\label{S_Ringel}
\subsection{Categorification on the Ringel dual}\label{SS_Ringel_cat}
The goal of this subsection is to obtain character formulas for tiltings in a basic categorification $\Cat$,
understand their images under the categorification functors, and produce a minimal tilting resolution
of each standard object. All this goals will be achieved once we equip the Ringel dual $\Cat^\vee$
with a highest weight categorical $\sl_2$-action.

First, let us recall a few standard things about Ringel duals  following Rouquier, \cite[4.1.5]{rouqqsch}. Let $\Cat$
be a highest weight category with poset $\Lambda$. Let $T$
be the sum of all indecomposable tiltings. Then $\Cat^\vee$ is the category of all finite dimensional
right $\End(T)$-modules. This category is highest weight with the poset opposite to $\Lambda$.
The  standard objects are $\Hom(T,\nabla(\lambda))$. There is an equivalence $\Cat^\nabla(=(\Cat^{opp})^\Delta)
\xrightarrow{\sim} (\Cat^\vee)^\Delta$ sending $M$ to $\Hom(T,M)$.  The tilting $T(\lambda)$
is sent to the projective $P^\vee(\lambda)$ in $\Cat^\vee$ corresponding to $\lambda$.
This induces an equivalence $\Cat^\vee-\operatorname{proj}\xrightarrow{\sim} \Cat-\operatorname{tilt}$.

Now suppose that $\Cat$ is equipped with an $\sl_2$-categorification such that $E,F$ preserve $\Cat^\Delta$.
Then $E,F$ preserve also $\Cat^{\nabla}$ and hence $\Cat$-$\operatorname{tilt}$. Using the identification $(\Cat^\vee)^{\Delta}= \Cat^\nabla$ we can transfer $E,F$ to
exact biadjoint functors $F^\vee,E^\vee$ on $(\Cat^\vee)^{\Delta}$ (so that $E$ corresponds to $F^\vee$). Since $\Cat^\vee-\operatorname{proj}=\Cat-\operatorname{tilt}$,
we see that $E^\vee, F^\vee$ preserve $\Cat^\vee-\operatorname{proj}$.
Being biadjoint, the functors $E^\vee,F^\vee$ uniquely extend to biadjoint functors on $\Cat^\vee$.
The transformations $X$ of $E^\vee$ and $T$ of $(E^\vee)^2$ are defined in an obvious way. So $E^\vee, F^\vee$
define a categorification. Clearly, $E^\vee,F^\vee$ preserve $(\Cat^{\vee})^{\Delta}$.

Now suppose that $\Cat$ is a highest weight $\sl_2$-categorification with respect to a
hierarchy structure on $\Lambda$. Let us show that $\Cat^\vee$  also becomes a highest
weight $\sl_2$-categorification if we modify the hierarchy structure on $\Lambda$.
For $t=(t_1,\ldots,t_n)\in \{+,-\}^n$, let $t^\vee=(\bar{t}_1,\bar{t}_2,\ldots, \bar{t}_n)$,
where $\bar{t}_i\neq t_i$. The  map $t\mapsto t^\vee$ is an order reversing bijection.
Define $\sigma_a^\vee:\{+,-\}^{n_a}\rightarrow \Lambda_a$ by $\sigma^\vee_a(t)=\sigma_a(t^\vee)$.
Then the collection $\Lambda_a,\sigma_a^\vee$ defines a family structure on $\Lambda$. Also
we can define a new splitting structure on $\Lambda$: $\Lambda^{\vee a}_>:=\Lambda^a_{<}, \underline{\Lambda}^{\vee a}_+:=\Lambda^a_{-},
\underline{\Lambda}^{\vee a}_-:=\underline{\Lambda}^a_+, \Lambda^{\vee a}_<:=\Lambda^{a}_>$. Similarly,
we define the whole hierarchy structure. The construction of $E^\vee,F^\vee$ together with Remark \ref{Rem:opp_cat}
imply that (ii) holds. Also each of (iii$^1$)-(iii$^3$) is preserved.

Clearly, if $\Cat$ is basic, then $\Cat^\vee$ is basic too. %For $t\in \{+,-\}^n$ we set
%$\Delta^\vee(t):=\Hom(T,\nabla)$

%Consider the case of a basic categorification $\Cat$. Then $\Cat^\vee$ is also basic.
%Set $\Delta^\vee(t):=\Hom(T, \nabla(\bar{t}))$. The category $\Cat^\vee$
%is highest weight with respect to this choice of standard objects. There is an equivalence $(\Cat^\Delta)^{opp}
%\xrightarrow{\sim} (\Cat^\vee)^\Delta$ sen

%The functors $E,F$ preserve both $\Cat^\Delta$ and $\Cat^\nabla$ and hence $\Cat-\operatorname{tilt}$.
%Using the identification $(\Cat^\vee)^{\Delta}= (\Cat^\Delta)^{opp}$ we can transfer $E,F$ to functors
%exact biadjoint functors $F^\vee,E^\vee$. Since $\Cat^\vee-\operatorname{proj}=(\Cat-\operatorname{tilt})^{opp}$,
%we see that $E^\vee, F^\vee$ preserve $\Cat^\vee-\operatorname{proj}$.
%
%Being biadjoint, the functors $E^\vee,F^\vee$ uniquely extend to biadjoint functors on $\Cat^\vee$.
%he transformations $X$ of $E$ and $T$ of $E^2$ are defined in an obvious way. So $E^\vee, F^\vee$
%define a categorification. To check that this categorification is highest weight. This reduces to
%computing the images of $E\Delta^\vee(t),F\Delta^\vee(t)$ in the Grothendieck group. This computation
%follows from the identification $(\Cat^\vee)^{\Delta}= (\Cat^\Delta)^{opp}$ that intertwines $E,F$
%and $F^\vee, E^\vee$ by the definition.

From this construction and Theorem \ref{Thm_proj_resol} we deduce in a straightforward way a description of a minimal tilting resolution
of a standard object and hence the character formulas for the indecomposable tiltings and the decomposition
of $E T(\lambda), F T(\lambda)$ into indecomposables. An analog of Proposition \ref{Prop:proj_arb} also holds. Let us record
it as we plan to use it in a subsequent paper.

\begin{Prop}\label{Prop:tilt_arb}
Let $\Cat$ be a highest weight $\sl_2$-categorification with respect to a hierarchy structure on a poset $\Lambda$.
Pick a family $\Lambda_a$ and $\lambda\in \Lambda_a$. Set $t:=\sigma_a^{-1}(\lambda)$. Then $ET(\lambda)$ contains
$T(\sigma_a(t^i))$ as a direct summand with multiplicity  $h_-(t_1,\ldots,t_{i})+1$ for all indices $i$ such that
$t_i=+, h_+(t_1,\ldots,t_{i-1})=0$ (and does not contain $T(\sigma_a(s))$ for the other elements $s$).
Similarly, $F T(\lambda)$ contains $T(\sigma_a(\bar{t}^i))$
with multiplicity $h_+^i(t)+1$ for all indices $i$ such that $h_-^{i+1}(t)=0$ and $t_i=-$.
\end{Prop}

\subsection{Reflection functor $\Theta$}
We are going to produce a concrete realization of the Ringel duality
for a basic categorification. For this we consider the reflection functor
$\Theta$ (the Rickard complex) originally defined by Rickard for symmetric groups, see \cite[Section 6.1]{CR}.
Let $\Cat$ be a  basic highest weight $\sl_2$-categorification of size $n$ and $\Cat=\bigoplus_{w=-n}^n\Cat_w$ be its weight decomposition.
Following \cite{CR}, for $d\geqslant 0$, consider the direct summands $E^{(sgn, d)}\subset
E^d, F^{(1,d)}\subset F^d$ that are realizations of divided powers $E^{(d)},F^{(d)}$ that has already appeared above.
Set $\Theta_w^{d}:=E^{(sgn, (n+w)/2-d)}F^{(1,(n-w)/2-d)}$ if $0\leqslant d\leqslant (n-|w|)/2$ (we remark that we shift the degrees
comparing to \cite{CR}). We set  $\Theta_w^{-d}:=0$ else.
Define a complex of functors $\Theta$ as follows: the restriction $\Theta$ to $\Cat_w$ is the complex $\ldots\rightarrow\Theta_w^i\rightarrow \Theta_w^{i+1}\rightarrow\ldots$, where the morphisms were constructed in \cite{CR}.

Our main result regarding a relationship between $\Theta$ and the Ringel duality is Proposition \ref{Prop:dual1}.

\subsection{Images of simples under the duality}
Let $T$ be a tilting generator of $\Cat$ and $L$ be a simple object in $\Cat$. In this subsection we are going
to determine $i$ such that $\Ext^i(T,L)\neq 0$. Recall that to $L$ we assign integers $\wt(L), d(L), h_-(L),h_+(L)$:
$\wt(L)=w$ if $L\in \Cat_w$, $h_-(L)$ (resp, $h_+(L)$) is the maximal integer $i$ such that $F^i L\neq 0$ (resp., $E^iL\neq 0$),
finally $d(L)=h_-(L)+h_+(L)+1$ -- this is the maximal dimension of a simple summand of the submodule of $[\Cat]$ generated
by $[L]$, in particular, $|\wt(L)|\leqslant d(L)-1\leqslant n$. The following claim is the main result of this subsection.

\begin{Prop}\label{Prop:tilt_vanish}
We have $\Ext^i(T,L)=0$ if $i\neq (n-d(L)+1)/2$.
%$i\leqslant (n-d(L)-1)/2$  or $i>(n-|\wt(L)|)/2-\delta$,
%where $\delta_{h_-(L)>0}=0$ if $h_-(L)=0$ or $h_+(L)=0$ and $1$ else.
\end{Prop}
\begin{proof}
We are going to prove this proposition by the induction on $n$ followed by the induction on $h_-(L)$.
We may assume that $\wt(L)\leqslant 0$ -- otherwise we can replace $\Cat$ with $\bar{\Cat}$.

{\it Step 1.} Let $\lambda,\mu$ be such that  $FL(\mu)=0$ and  $\Ext^i(T(\lambda), L(\mu))\neq 0$ for some $i$. %We remark
%that $\lambda$ does not start with a $-$. Indeed, by the previous subsection, otherwise $T(\lambda)$ is
%a direct summand in $ET$ and hence $\Ext^i(T(\lambda), L(\mu))$ is a direct summand of
%$\Ext^i(ET, L(\mu))=\Ext^i(T, FL(\mu))=0$.
So $\mu$ has the form $t+$ for some $t\in \{+,-\}^{n-1}$. Let us consider the cases $\lambda=s+, \lambda=s-$
separately.

{\it Step 2.} Assume $\lambda=s-$. Recall the subcategory $\Cat_-\subset \Cat$ and the quotient category
$\Cat_+$ of $\Cat$ that are equipped with isomorphic highest weight $\sl_2$-categorifications.
Then  $T(s-)\in \Cat_-$. Let $T_+(s)$ be the image of $T(s-)$ under the identification $\Cat_-\cong \Cat_+$.
This is a tilting in $\Cat_+$. We view $T_+(s)$ as an object in $\Cat$ via the inclusion $\Cat_+^\Delta\subset \Cat^\Delta$.
As we have noticed in the previous section, we have a short exact sequence
$0\rightarrow E_+ T_+(s)\rightarrow E T_+(s)\rightarrow T(s-)\rightarrow 0$. Since $\Ext^i(E T_+(s),L(t+))=0$
for all $i$ (this is because $FL(t+)=0$) we see that $\Ext^i(T(s-),L(t+))=\Ext^{i-1}(E_+ T_+(s), L(t+))$ for all $i$. The right hand side
is computed in $\Cat$, but, since $E_+T_+(s)\in \Cat_+^\Delta$, in $\Cat_+$ we have the same result (where, rigorously speaking,
we need to replace $L(t+)$ with its image $L_+(t)$). If $F_+ L_+(t)=0$, then $\Ext^\bullet(E_+ T_+(s), L_+(t))=\Ext^{\bullet}(T_+(s), F_+ L_+(t))=0$. Otherwise $h_-(t)=1$ so $F_+ L_+(t)$ is simple with $d(L_+(t))=d(L(t+))+1$. Using the induction
assumption on $n$ we see that $\Ext^i(T_+(s), F_+L_+(t))=0$ whenever $i\neq(n-1+1-d(L)-1)/2$. So
$\Ext^i(T(s-), L(t+))=0$ whenever $i\neq (n+1-d(L))/2$.

{\it Step 3.} Now assume $\lambda=s+$. Then $T(s+)$ is a direct summand in $F T(s-)$, thanks to Proposition \ref{Prop:tilt_arb}.
So $\Ext^i(T(s+), L(t+))\neq 0$
implies $$\Ext^i(F T(s-), L(t+))=\Ext^i(T(s-), EL(t+))\neq 0.$$ The simple subquotients of $EL(t+)$ are
$L(t-)$ and some simples $L_0$ with $FL_0=0$ (and hence $d(L_0)$ is either $d(L_0)-2$ or $2$
if $d(L)=2$). By the previous step, we have $\Ext^i(T(s-),L_0)=0$
whenever $i\neq (n-d(L_0)+1)/2$. Also $\Ext^i(T(s-),L(t-))$ can be computed in $\Cat_-$, compare with the proof
of Lemma \ref{Lem:ind_t-}. We know by induction
on $n$ that the ext may be nonzero only if $i=(n-1+1-d(L_-(t)))$. But $d(L_-(t))=h_+(t)+h_-(t)+1= d(L)-1$.
So $\Ext^i(T(s-), L(t-))=0$ if $i\neq (n+1-d(L))/2$. It follows that $\Ext^i(ET(s-), L(t+))$ can be nonzero
only if $i\neq (n+1-d(L))/2,(n+1-d(L_0))/2$. As we have seen $d(L_0)$ is either $d(L)-2$ or $d(L)$.
In the former case, $(n+1-d(L_0))/2>(n+\wt(L))/2$.
Let us show that $\Ext^{i}(T,L)=0$ for $i>(n+\wt(L))/2$. Indeed, Theorem \ref{Thm_proj_resol} shows
that $\Ext^i(\Delta(t'),L)=0$ for $i>(n+\wt(L))/2$ and any $t'\in \{+,-\}^n$ (this is because $s(D)\leqslant
(n-|\wt(t')|)/2$ for all divisions $D$; of course, if $\wt(t')\neq \wt(L)$, then all Ext's are zero).
Since $T$ is standardly filtered, we see that $\Ext^i(T,L)=0$ for $i>(n+\wt(L))/2$. This completes the
proof of our claim in the case when $FL=0$.

{\it Step 4.} Consider now the general case: $FL\neq 0$. Then $L$ is both the head
and the socle of $EL'$, where $L'$ is a simple with $d(L')=d(L), h_-(L')=h_-(L)-1$. So, by the inductive assumption,
we have $\Ext^i(FT,L')=0$ if $i\neq (n-d(L)+1)/2$  because $FT$ is tilting. Since $E$ and $F$
are biadjoint, we have $\Ext^i(FT,L')=\Ext^i(T,EL')$. Remark \ref{Rem:part2_equiv} implies  that any simple subquotient
$L_0$ of $EL'$ different from $L$ satisfies $h_-(L_0)=h_-(L)-1$. So we have $\Ext^i(T,L_0)=0$ if
$i\neq (n-d(L)+3)/2$.
From here we deduce that $\Ext^i(T,L)=0$ implies $\Ext^i(T,K)=0$ for
$i\neq (n-d(L)+3)/2$ and any subquotient $K$ of $EL'$.

Let us take the kernel of the natural epimorphism $EL'\twoheadrightarrow L$ for $K$.
Then we have an exact sequence
$$\Ext^{i}(T, EL')\rightarrow \Ext^{i}(T,L)\rightarrow \Ext^{i+1}(T,K)\rightarrow \Ext^{i+1}(T,EL').$$
It implies that $\Ext^i(T,L)=\Ext^{i+1}(T,K)$ for $i\neq  (n-d(L)\pm 1)/2$. If $i\neq (n-d(L)+1)/2$,
then $\Ext^{i+1}(T,L)=0$ implies $\Ext^{i+1}(T,K)=0$. Since $\Ext^{i+1}(T,L)=0$ for
all $i$ sufficiently large, we see that $\Ext^i(T,L)=0$ for $i> (n-d(L)+1)/2$.

Now let us take the cokernel of the embedding $L\rightarrow EL'$ for $K$. Then  we  have the exact sequence
$$\Ext^{i-1}(T,EL')\rightarrow \Ext^{i-1}(T,K)\rightarrow \Ext^i(T,L)\rightarrow \Ext^i(T,EL'),$$
where we set $\Ext^{-1}(\bullet,\bullet)=0$. Arguing similarly to the previous paragraph we
see that $\Ext^i(T,L)=0$ for $i\leqslant (n-d_-(L)-1)/2$.

\end{proof}

\subsection{The main result}
\begin{Prop}\label{Prop:dual1}
The functor $\operatorname{RHom}(T, \Theta^{-1}\bullet):\Cat\rightarrow D^b(\Cat^\vee)$
is exact and so is an equivalence of abelian categories $\Cat\rightarrow \Cat^\vee$.
\end{Prop}

For the standard basic categorification  this follows, for example, from results
of \cite{MS}.

\begin{proof}
We claim that $\operatorname{RHom}(T, \Theta^{-1}(E^{(k)}L))$ has only $0$th cohomology
as long as $L$ is a simple such that $FL=0$. If $k=0$, then $\Theta^{-1}$
maps $L$ to $E^{(\wt(L))}L[-\wt(L)]$,
and Proposition \ref{Prop:tilt_vanish} shows that $\operatorname{RHom}(T, E^{(\wt(L))}L[-\wt(L)])$
has only $0$th cohomology. If $k>0$, then the proof of Proposition 5.4 in \cite{CKL} shows that
$\Theta^{-1}(E^{(k)}L)$ is isomorphic to $F^{(k)}\Theta^{-1}(L)$ (recall that our functors are shifted
comparing to \cite{CR}). Also $\operatorname{RHom}(T, F^{(k)}\Theta^{-1}(L))=
\operatorname{RHom}(E^{(k)}T, \Theta^{-1}(L))$. Since $E^{(k)}T$ is tilting, we see that
$\operatorname{RHom}(E^{(k)}T, \Theta^{-1}(L))$ is again concentrated in homological degree $0$.

Let $\mathcal{F}_i$ be the $i$th cohomology of $\operatorname{RHom}(T, \Theta^{-1}\bullet)$. Let us choose the maximal $i$
such that $\mathcal{F}_i$ is nonzero. We need to show that $i=0$.
Assume $i>0$. The functor $\mathcal{F}_i$ is left exact. The previous paragraph shows
that  $\mathcal{F}_i$ vanishes on all objects of the form $E^{(k)}L$, where $L$ is a simple with $FL=0$.
But any simple embeds into an object of that form. Since $\mathcal{F}_i$ is left exact, this
means that $\mathcal{F}_i$ vanishes on any simple, and hence on any object.

To prove that $\mathcal{F}_i$ vanishes for $i<0$ we argue similarly.
\end{proof}

%\section{Further structural results}\label{S_struct_futh}

\subsection{Homs between standards}
We conclude this section by  showing that the Homs  between standard modules in a basic categorification
behave in the same way as in the standard basic categorification.

The main result of this subsection is the following proposition.

\begin{Prop}\label{Prop:stand_hom}
Let $\Cat$ be a basic highest weight $\sl_2$-categorification of size $n$. Then for $t,s\in \{+,-\}^n$
we have $\dim \Hom(\Delta(s),\Delta(t))=1$ if $s\leqslant t$ and $\Hom(\Delta(s),\Delta(t))=0$ else.
Moreover, any morphism of standards is injective.
\end{Prop}

Recall that $s\leqslant t$ means that for any $j$ the number of $-$'s among $s_1,\ldots,s_j$
does not exceed that among $t_1,\ldots,t_j$, with the equality for $j=n$.

The proof requires several lemmas.  First, using Proposition \ref{Prop:tilt_vanish} we can describe
the simples occurring in the socles of standards.

\begin{Lem}\label{Cor:stand_socle}
If $L(s)$ occurs in the socle of $\Delta(t), t\in \{+,-\}^n$, then $d(s)=n+1$, i.e.,
all $+$'s in $s$ occur to the left of $-$.
\end{Lem}
\begin{proof}
A simple occurs in the socle of a standard if and only if it occurs in the head of a costandard in $\Cat^{opp}$, equivalently,
if it occurs in the head of a tilting. Now the claim of the corollary follows from Proposition \ref{Prop:tilt_vanish}.
\end{proof}

\begin{Lem}\label{Lem:socle_mult}
Let $t\in \{+,-\}^n$ and $s$ be such that $\wt(s)=\wt(t)$ and $d(s)=n+1$. Then the multiplicity of
$\Delta(s)=L(s)$ in $\Delta(t)$ is $1$.
\end{Lem}
\begin{proof}
By the BGG reciprocity, $[\nabla(t):L(s)]=[P(s):\Delta(t)]$. By Proposition \ref{Prop:character},
$P(s)=E^{(\wt(s))}\Delta(+\ldots+)$. So the multiplicity of any standard $\Delta(t)$ with $\wt(t)=\wt(s)$ in $P(s)$ is $1$.
It follows that the multiplicity of $L(s)$ in any $\nabla(t)$ is $1$. Applying this to $\Cat^{opp}$, we get the required
result.
\end{proof}

%\begin{Lem}\label{Lem:Hom_subquot}
%Let $\Cat$ be a highest weight category with poset $\Lambda$.
%Choose subsets $\Lambda_1\subset \Lambda_2\subset\Lambda$ with the property that $\lambda\in \Lambda_i, \mu<\lambda$
%implies $\mu\in \Lambda_i$ for $i=1,2$.  Let $\lambda,\mu\in \Lambda_2\setminus \Lambda_1$.
%Consider the highest weight subcategories $\Cat(\Lambda_i)\subset \Cat$ and the quotient $\Cat'=\Cat(\Lambda_2)/\Cat(\Lambda_1)$
%with quotient functor  $\pi:\Cat(\Lambda_2)\rightarrow \Cat'$.
%Then $\Hom_{\Cat}(\Delta(\lambda),\Delta(\mu))\cong \Hom_{\Cat'}(\pi\Delta(\lambda),\pi\Delta(\mu))$.
%\end{Lem}
%\begin{proof}
%Of course, the Hom spaces between $\Delta(\lambda)$ and $\Delta(\mu)$ in $\Cat$ and in $\Cat(\Lambda_2)$ coincide. But also %$\Hom_{\Cat(\Lambda_2)}(\Delta(\lambda),\Delta(\mu))$ is naturally isomorphic to $\Hom_{\Cat'}(\pi(\Delta(\lambda)), \pi(\Delta(\mu)))$. The %proof of this fact basically was contained in the proof of Lemma \ref{Lem:Eun_stand}, where we have essentially checked that %$\pi^!\pi\Delta(\lambda)=\Delta(\lambda)$.
%\end{proof}

\begin{proof}[Proof of Proposition \ref{Prop:stand_hom}]
Lemmas \ref{Cor:stand_socle},\ref{Lem:socle_mult} imply that $\dim \Hom(\Delta(s),\Delta(t))\leqslant 1$
and that any nonzero morphism is injective. It remains to show that $\Hom(\Delta(s),\Delta(t))\neq 0$ when $s<t$.
We remark that for any $s<t$ there is a sequence $t^0=t, t^1,\ldots, t^m=s$ with the following properties:
\begin{itemize}
\item $t^m<t^{m-1}<\ldots<t^1<t^0$.
\item $t^{j+1}$ is obtained from $t^j$ by switching a single consecutive pair consisting of a $+$ and a $-$.
\end{itemize}
So it is enough to show that $\Hom(\Delta(s),\Delta(t))\neq 0$, when there is an index $j$
such that $s_i=t_i$ for $i\neq j,j+1$, $s_j=t_{j+1}=+, s_{j+1}=t_j=-$. By Lemma \ref{Lem:hw_quot},
the Hom spaces between standards with labels in an interval are the same in the ambient category
and in the highest weight subquotient corresponding to the interval.
Using the (usual and dual) categorical splittings from Subsection \ref{SS:cat_pm},
we reduce the question to $n=2$. Clearly, $\dim\Hom(\Delta(+-),\Delta(-+))=1$.
\end{proof}

\begin{Rem}\label{Rem:stand_hom}
Thanks to the existence of a family filtration,
Proposition \ref{Prop:stand_hom} together with the standard results recalled in Subsection \ref{SS_hw_reminder}  allow to compute
$\Hom(\Delta(\lambda),\Delta(\mu))$ in an arbitrary highest weight $\sl_2$-categorification
whenever $\lambda$ and $\mu$ are in the same family. This will be applied in the next section
to the case of categories $\mathcal{O}$ over cyclotomic Rational Cherednik algebra.
\end{Rem}

\section{Applications to cyclotomic rational Cherednik algebras}\label{S_Cher_appl}
\subsection{$\Hom(\Delta(\operatorname{refl}), \Delta(\operatorname{triv}))$}
 Consider the cyclotomic Cherednik  category $\mathcal{O}=\bigoplus_{n=0}^{+\infty}\mathcal{O}(n)$ depending on  parameters $(\kappa, s_0,\ldots, s_{\ell-1})$, where $\kappa$ is non-integral. This category carries multiple highest weight
 categorical $\sl_2$-actions, \cite[3.5]{cryst}.  The standards are parameterized by
the poset of  $\ell$-multipartitions recalled in Subsection \ref{SS_hier_ex}. Such a multi-partition $\lambda$ can be thought
as a representation of $G(n,1,\ell)$, where $n=|\lambda|$. In particular, we have the trivial representation $\operatorname{triv}$
corresponding to the multipartition $\tau_n:=(n,0,0,\ldots,0)$ and the reflection representation corresponding to
$\rho_n:=(n-1,1,0,\ldots,0)$.

In \cite{FS} Feigin and Silantyev studied the space $\Hom(\Delta(\operatorname{refl}), \Delta(\operatorname{triv}))$
in the case of Cherednik algebras of Coxeter groups and so called ``equal parameters'' (they also got some
partial results for the groups $G(n,1,\ell)$).  This study was related to the theory of Frobenius manifolds.

Using the techniques developed above we will describe the spaces $\Hom(\Delta(\operatorname{refl}), \Delta(\operatorname{triv}))$
in the case of cyclotomic groups completely.

\begin{Prop}\label{Prop:stand_hom_Cher}
The space $\Hom(\Delta(\rho_n),\Delta(\tau_n))$ is one-dimensional provided $\kappa(s_0-s_1+n-1)$ is a negative integer
and is zero else.
\end{Prop}
\begin{proof}
Let $b$ denote the box $(1,n,0)$ and $b'$ denote the box $(1,1,1)$ so that $\tau_n=(\tau_n\cap \rho_n)\sqcup b,
\rho_n=(\tau_n\cap \rho_n)\sqcup b'$.
The necessary condition for the Hom space to be nonzero is that $\Delta(\rho_n),\Delta(\tau_n)$ lie
the same block and $\rho_n<\tau_n$ with respect to the highest weight ordering. The first condition implies that
for any possible pair $E_z,F_z$ of the categorification functors the corresponding weight functions of $\tau_n,\rho_n$
coincide. This boils down to $b\sim b'$, i.e., $\kappa(n+s_0-1-s_1)\in \Z$. The condition $\rho_n<\tau_n$ is equivalent
to $b'<b$, i.e., to $\kappa \ell (s_0+n-1)<\kappa \ell s_1-1$. The conditions $\kappa(n+s_0-1-s_1)\in \Z$
and  $\kappa \ell (s_0+n-1)<\kappa \ell s_1-1$ together are equivalent to  $\kappa(s_0-s_1+n-1)\in \Z_{<0}$.

Now suppose that $\kappa(s_0-s_1+n-1)\in \Z_{<0}$. The boxes $b,b'$ have the same residue modulo $\kappa^{-1}$.
Consider the $\sl_2$-categorification corresponding to that residue. Then $\tau_n,\rho_n$ belong to the same
family, say $\Lambda_a$. Let $t,s\in \{+,-\}^{n_a}$ be such that $\sigma_a(t)=\tau_n, \sigma_a(s)=\rho_n$.
Then there is only one $-$ in both $t,s$.

Thanks to Remark \ref{Rem:stand_hom},
$\Hom(\Delta(\rho_n),\Delta(\tau_n))$ is one-dimensional.
\end{proof}

We finish this subsection by considering the case of $\ell=2$ in more detail
and relating our result to \cite{FS}. Here one can also
use parameters $c_1,c_2$ corresponding to the conjugacy classes of reflections in $G(n,1,2)$,
$c_1$ to the reflections contained in the symmetric group and $c_2$ to reflections from $\Z/2 \Z$.
They are related to $\kappa,s_0,s_1$ by the following formulas, see, e.g., \cite{cryst}.
$$c_1=-\kappa, c_2=\kappa(s_1-s_0)-\frac{1}{2}.$$
In particular, the equal parameter case, $c_1=c_2$, corresponds to $\kappa(s_1+1-s_0)=\frac{1}{2}$.
In this case, the condition of the proposition becomes $\kappa n-\frac{1}{2}\in \Z_{<0}$ and, for $c_1>0$, we recover
the condition from \cite{FS} for the Weyl groups of type $B$: $c_1=\frac{m}{2n}$, where $m$
is a positive  odd number.

\subsection{Blocks}
In this subsection we are going to describe blocks of the category $\OCat$ for cyclotomic Rational Cherednik algebras.
We remark that the description is known thanks to results of Lyle and Mathas, \cite{lylemat}, who described
blocks of the cyclotomic Hecke algebras. The description for $\OCat$ follows from that and the properties of
the KZ functor, compare to the proof of \cite[Lemma 5.16]{shanvasserot}. We are going to obtain an independent proof.

Recall the block equivalence relation on the set of simples. We say that simple objects $L,L'$ lie in the same block
if there are simples $L_0=L, L_1,\ldots, L_k=L'$ such that for each $i=1,\ldots,k-1$ we have $\Ext^1(L_{i-1},L_i)\neq 0$
or $\Ext^1(L_i,L_{i-1})\neq 0$. In this case we will write $L\sim_b L'$. If $L=L(\lambda), L'=L(\lambda')$
we write $\lambda\sim_b \lambda'$.

%To a multipartition $\lambda$ we can assign an unordered $|\lambda|$-tuple $\res(\lambda)$ of elements
%of $\C/\kappa^{-1}\Z$. By definition, $\res(\lambda)$ consists of elements $\cont(b), b\in \lambda,$
%viewed modulo $\kappa^{-1}$.

On the other hand, to each $\lambda\in \mathcal{P}_\ell$ we can assign a collection $(\wt_z(\lambda))_{z\in \C/\kappa^{-1}\Z}$
of the weight of $[\Delta(\lambda)]$ with respect to the $\sl_2$-categorification associated to $\lambda$.
We write $\lambda\sim_w\lambda'$ if the collections for $\lambda$ and $\lambda'$ coincide and also $|\lambda|=|\lambda'|$.

Here is the main result of this subsection.

\begin{Prop}\label{Prop:block}
The equivalence relations $\sim_b$ and $\sim_w$ on $\mathcal{P}_\ell$ coincide.
\end{Prop}

It is clear that $\lambda\sim_b \lambda'$ implies $ \lambda\sim_w \lambda'$. Also let us remark that there is an alternative
description of the equivalence relation $\sim_w$. Namely, to a multipartition $\lambda$ we can assign its residue $\operatorname{res}(\lambda)$, the multi-set of the residues of all boxes in that multi-partition. It is known (and not difficult
to check) that $\lambda\sim_w \lambda'$ if and only if $\res(\lambda)=\res(\lambda')$.

We start by noticing that in a basic categorification $L(t)$ and $L(s)$ lie in the same block if and only if
$\wt(t)=\wt(s)$.  Proposition \ref{Prop:stand_hom} together with Remark \ref{Rem:stand_hom} imply that if $\lambda,\mu$ lie in the same
$z$-family and $\wt_z(\lambda)=\wt_z(\mu)$, then $\lambda\sim_b\mu$. Indeed, in this case $\Delta(\lambda),\Delta(\mu)$
admit a nonzero morphism from the same standard object. This observation motivates us to introduce  yet another equivalence relation on $\mathcal{P}_\ell$. We write $\lambda\sim_f \lambda'$ if there is a sequence of
elements $\lambda_0=\lambda, \lambda_1,\ldots,\lambda_k=\lambda'$ with the following properties: there are
$z_1,\ldots,z_{k}\in \C/\kappa^{-1}\Z$ such that $\lambda_{i-1}$ and $\lambda_i$ lie in the same $z_i$-family
and have the same $z_i$-weight for all $i=1,\ldots,k$. The above discussion shows that $\lambda\sim_f \lambda'$
implies  $\lambda\sim_b \lambda'$.

\begin{Lem}\label{Lem:blocks_ind}
Let $\Cat$ be a highest weight $\sl_2$-categorification. Let $L,L'$ be simples such that $\tilde{e}^k L,\tilde{e}^k L'\neq 0$
for some positive number $k$. If $\Ext^1(L,L')\neq 0$, then $\tilde{e}^k L\sim_b \tilde{e}^k L'$.
\end{Lem}
\begin{proof}
We can assume that $L\not\leqslant L'$ in the highest weight order, otherwise we can replace $\Cat$ with $\Cat^{opp}$
switching $L$ and $L'$. Let $M$ be a nontrivial extension of $L$ by $L'$. Let $\Delta_L$ be the standard object with head $L$.
Then $\Delta_L$ is projective in the Serre subcategory of $\Cat$ spanned by all $L''\not> L$.
In particular, we have a nonzero morphism $\Delta_L\rightarrow M$. Since the extension is non-trivial,
the head of $M$ is also $L$. So the morphism $\Delta_L\rightarrow M$ is surjective.

It follows that the morphism $E^{(k)}\Delta_L\rightarrow E^{(k)}M$ is surjective. But $E^{(k)}\Delta_L$ is filtered with
standard subquotients that belong to the same family as $L$. So all simple constituents of $E^{(k)}\Delta_L$
lie in the same block as $\tilde{e}^k L$. But $\tilde{e}^k L'$ is one of these subquotients.
\end{proof}

\begin{proof}[Proof of Proposition \ref{Prop:block}]
Assume that $\lambda\sim_w\lambda'$ implies $\lambda\sim_b\lambda'$ for $|\lambda|<n$
and let us prove the claim for $|\lambda|=n$. Assume the converse: there are $\lambda,\lambda'$
with $|\lambda|=n$ and $\lambda\sim_w\lambda'$ such that $\lambda\not\sim_b\lambda'$.

First, we claim that $\wt_z(\lambda)\leqslant 0$ for all $z$. Indeed, otherwise elements
$\mu:=\tilde{f}_z^{\wt_z(\lambda)}\lambda, \mu':=\tilde{f}_z^{\wt_z(\lambda)}\lambda'$
are nonzero. We have $\mu\sim_w \mu'$ and $|\mu|<|\lambda|$. So $\mu\sim_b \mu'$. Moreover, for any element
$\mu''$ with $\mu''\sim_w \mu$ we have $\tilde{e}_z^{\wt_z(\lambda)}\mu''\neq 0$. Take a
collection $\mu_0:=\mu, \mu_1,\ldots, \mu_k:=\mu'$ such that, for each $i$, $\Ext^1(L(\mu_i),L(\mu_{i+1}))\neq 0$
or $\Ext^1(L(\mu_{i+1}),L(\mu_i))\neq 0$.
Lemma \ref{Lem:blocks_ind} implies that $\tilde{e}_z^{\wt_z(\lambda)} \mu_i\sim_b \tilde{e}_z^{\wt_z(\lambda)}\mu_{i+1}$
for all $i$. Hence $\lambda\sim_b\lambda'$.

Also Lemma \ref{Lem:blocks_ind} implies that $\lambda$ and $\lambda'$ cannot simultaneously
have removable $z$-boxes, for any $z$. Indeed, otherwise we can find $\tilde{\lambda}\sim_f \lambda,
\tilde{\lambda}'\sim_f \lambda'$ such that $h_{z,-}(\tilde{\lambda}),h_{z,-}(\tilde{\lambda}')>0$.
Then we can apply Lemma \ref{Lem:blocks_ind} to $L(\tilde{f}_z \tilde{\lambda}),L(\tilde{f}_z\tilde{\lambda}')$
and, since $\tilde{e}_z L\neq 0$ for any simple $L$
in the block of $L(\tilde{f}_z \tilde{\lambda})$, get that $\tilde{\lambda}\sim_b\tilde{\lambda}'$. Therefore,
by the discussion following the statement of Proposition \ref{Prop:block}, we have $\lambda\sim_b\lambda'$.

Now consider the case when $\lambda\cap \lambda'\neq \varnothing$. Let $i$ be the index of
a partition, where $\lambda$ and $\lambda'$ have common boxes. We will move boxes from
$\lambda$ using the following recipe. Start with the topmost removable box in the $i$th partition. By our assumption
on the weights of $\lambda$ there should be an addable box with the same residue. Place the removable
box to the place of that addable box.  This operation definitely does not change the class with respect
to $\sim_f$. If we choose the addable box below the removable one
or in a different partition, this operation does not cycle. We continue the operation until only one box
is left in the $i$th partition. If we can do this for both $\lambda$ and $\lambda'$, then we finish
with two multi-partitions, say $\tilde{\lambda},\tilde{\lambda'}$ that share the same removable box.

So let us consider the situation when the algorithm cycles. I.e., on some step we should encounter the situation
where we can only move the removable box to a position above in the same diagram. This is only possible when $\kappa$ is rational,
otherwise two boxes in the same diagram cannot have the same residue. So let $\kappa=\frac{p}{q}$, where $p,q$
are coprime.

Suppose we have this situation with $\lambda$. This means that on some step, for the diagram $\lambda^{(i)}=(\lambda_1^{(i)},\ldots, \lambda_m^{(i)})$, where $m$ is maximal with $\lambda_m^{(i)}\neq 0$, the number  $\lambda_m^{(i)}$ is divisible by $q$
(otherwise the box with coordinates $(m+1,1)$ has different residue from the box with coordinate $(m,\lambda_m^{(i)})$).
Let $\alpha$ be the residue (modulo $\kappa^{-1}$) of the rightmost box in the $m$th row.
Then we can move boxes from the $m$th row to the first column. In this way we subsequently get removable boxes
with residues $\alpha- i$ for $i=0,\ldots,\lambda_m^{(i)}-2$. In other words, we realize all possible residues that can occur
in the $i$th diagram with exception of one if $\lambda_m^{(i)}=q$, and that single residue is $\alpha-q+1$.

If $\lambda'^{(i)}$ also cycles and $q>2$, then we are done. If $\lambda'^{(i)}$ does not cycle and also initially
has more than one box, then we are also done, as in the process removable boxes with at least two different residues occur.
So we need to consider two situations: either 
\begin{itemize} \item $q=2$ and $\lambda'^{(i)}$ cycles or 
\item $\lambda'^{(i)}$ consists of a single box.\end{itemize}

Suppose that $\lambda'^{(i)}$ consists of a single box. The residue has to be $\alpha-q+1$. If in $\lambda'$
we have a removable box with residue $\alpha$, then we are done (getting removable boxes with same residues in
$\lambda$ and $\lambda'$). So we have only addable $\alpha$-boxes in $\lambda'$ and therefore $\wt_{\alpha}(\lambda)=\wt_{\alpha}(\lambda')>0$ (there is an addable box with residue $\alpha$ in $\lambda'$). So there should be more than one addable box with
residue $\alpha$ in $\lambda$ (because there is removable box with residue $\alpha$) and therefore our
algorithm for $\lambda$ does not cycle.

Now consider the case when $q=2$ and $\lambda'$ also cycles. There are only two possible residues for the boxes
in the $i$th diagram: $\alpha,\alpha+1$. In $\lambda^{(i)}$ (resp., $\lambda'^{(i)}$) all removable boxes have
residue $\alpha$ (resp., $\alpha+1$). Also there is only one addable box with residue $\alpha$ (resp., $\alpha+1$)
in $\lambda$ (resp., $\lambda'$) and this box is in the $i$th partition because our algorithm cycles for both $\lambda$ and $\lambda'$.
Since $\wt_\alpha(\lambda)\geqslant 0$, we see that there is only one removable box in $\lambda^{(i)}$.
The only possibility here is that $\lambda^{(i)}=(2)$. Similarly, $\lambda'^{(i)}=(2)$. A contradiction.
We have completed the case when $\lambda\cap\lambda'\neq \varnothing$.

Now consider the case when $\lambda\cap\lambda'=\varnothing$. We can also assume that if $\tilde{\lambda}\sim_f \lambda,
\tilde{\lambda}'\sim_f \lambda'$, then $\tilde{\lambda}\cap \tilde{\lambda}'=\varnothing$.

Pick an index $i$ with $\lambda^{(i)}\neq \varnothing, \lambda'^{(i)}=\varnothing$. Let $\alpha$ be the residue of
the box with coordinates $(0,0)$ in the $i$th diagram. Then $\lambda'$ has an addable box in  that position and no removable boxes
with that residue. But $\lambda$ has boxes with residue $\alpha$ and, since $\res(\lambda)=\res(\lambda')$,
so does $\lambda'$, say in the $j$th partition.

Now we can apply the same box moving algorithm to $\lambda'^{(j)}$ as before. We never encounter a
removable $\alpha$-box so the algorithm cycles. Then, again, $\kappa=\frac{p}{q}$ and in $\lambda'^{(j)}$ on various steps
we can have boxes with residues $\alpha+1,\ldots\alpha+q-1$. But we can also apply our algorithm to $\lambda^{(i)}$.
As above we conclude that $q=2$ and in $\lambda'^{(j)}$ we only have removable boxes with residues $\alpha+1$,
while in $\lambda^{(j)}$ we only have removable boxes with residues $\alpha$ and only one addable box with residue
$\alpha$. This forces $\lambda'$ to have a removable $\alpha$-box. We are done.
%It remains to prove that there are $\tilde{\lambda}\sim_f\lambda$ and $\tilde{\lambda}'\sim_f \lambda'$
%such that $\tilde{\lambda},\tilde{\lambda}'$ have a removable box with the same residue.
\end{proof}

\section{Subsequent developments}
In this section we describe some related developments that took place after this paper was first made public.
\subsection{Categorifications of tensor products}
Basic highest weight categorifications for $\sl_2$ studied in this paper  are special cases of tensor product categorifications
considered in \cite{LW} (when all tensor factors are the tautological $\sl_2$-modules). The main result
of \cite{LW} is a uniqueness theorem for tensor product categorifications. In particular, this theorem
shows that a basic categorification of size $n$ is unique and therefore yields an alternative proof
of results in Sections 6,7 in the case when the characteristic is 0. We remark that splitting techniques
partially generalizing those from Section 5 play an important role in \cite{LW}. In the positive characteristic case
the results of Sections 6,7 still seem to be new.

\subsection{Categorifications of higher level Fock spaces}
Another setting when one would want a uniqueness result is when  the algebra of interest is $\hat{\sl}_e$
and the module being categorified is a level $\ell$ Fock space. Here an example of a categorification is
given by cyclotomic Cherednik categories $\mathcal{O}$.
Another class of categorifications arise from parabolic affine categories $\mathcal{O}$, see \cite{Cher_mult}
for details. A concrete application of a uniqueness theorem would be a proof of a conjecture by Varagnolo
and Vasserot, \cite{VV}, relating Cherednik categories $\mathcal{O}$ to affine parabolic categories $\mathcal{O}$.
The Varagnolo-Vasserot conjecture was recently proved in \cite{RSVV},\cite{Cher_mult}. There is a more general
uniqueness result obtained in the latter paper but it requires using some other structures in addition
to a categorical action and a highest weight structure. The proof again extensively uses results
and constructions of the  present paper.

\end{document}